\documentclass[12pt]{amsart}
\usepackage{amssymb}
\usepackage{amscd}
\usepackage{amsthm}
\usepackage[dvips]{graphicx}
\usepackage[all,cmtip]{xy}
\usepackage{hyperref}

\usepackage{color}
\usepackage{ulem}

\usepackage[all]{xy}

\newtheorem{Thm}{Theorem}[section]
\newtheorem{Lem}[Thm]{Lemma}
\newtheorem{Def}[Thm]{Definition}
\newtheorem{Cor}[Thm]{Corollary}
\newtheorem{Prop}[Thm]{Proposition}
\newtheorem{Ex1}[Thm]{Example}
\newtheorem{Rem1}[Thm]{Remark}

\newtheorem{Prob}[Thm]{Problem}

\newenvironment{Rem}{\begin{Rem1}\rm}{\end{Rem1}}
\newenvironment{Ex}{\begin{Ex1}\rm}{\end{Ex1}}

\title[Simple-minded systems]{Simple-minded systems, configurations and
mutations for representation-finite self-injective algebras}
\author{Aaron Chan, Steffen Koenig, and Yuming Liu$^*$}

\address{Aaron Chan
\newline Institute of Mathematics
\newline University of Aberdeen
\newline Aberdeen, AB24 3UE
\newline UK}
\email{aaron.kychan@gmail.com}

\address{Steffen Koenig
\newline Institute of
algebra and number theory
\newline University of Stuttgart
\newline Pfaffenwaldring 57
\newline 70569 Stuttgart
\newline Germany}
\email{skoenig@mathematik.uni-stuttgart.de}

\address{Yuming Liu
\newline School of Mathematical Sciences
\newline Laboratory of Mathematics and Complex Systems
\newline Beijing Normal University
\newline Beijing 100875
\newline P.R.China}
\email{ymliu@bnu.edu.cn}

\date{version of \today}

\newenvironment{Proof}[1][Proof]{\begin{trivlist}
\item[\hskip \labelsep {\bfseries #1}]}{\flushright
$\Box$\end{trivlist}}

\newcommand{\lra}{\longrightarrow}

\newcommand{\ra}{\rightarrow}
\newcommand{\sdp}{\times\kern-.2em\vrule height1.1ex depth-.05ex}
\newcommand{\epi}{\lra \kern-.8em\ra}

\newcommand{\Z}{{\mathbb Z}}

\newcommand{\stmod}{\underline{\mathrm{mod}}}
\newcommand{\al}{\alpha}
\newcommand{\be}{\beta}

 \normalbaselines
\thanks{The authors would like to thank Alex Dugas and Dong Yang for
enlightening discussions, and Hideto Asashiba for
kindly sending us the source code of the representative quiver of the
RFS algebra of type $(D_{3m},s/3,1)$. The first named author is
supported by EPSRC Doctoral Training Scheme. The third named author
has been supported by a research fellowship from Alexander von Humboldt
Foundation and partially by NCET Program from MOE of China.}

\begin{document}
\renewcommand{\thefootnote}{\alph{footnote}}
\setcounter{footnote}{-1} \footnote{ $^*$ Corresponding author.}
\renewcommand{\thefootnote}{\alph{footnote}}
\setcounter{footnote}{-1} 
\renewcommand{\thefootnote}{\alph{footnote}}
\setcounter{footnote}{-1} 

\begin{abstract}
Simple-minded systems of objects in a stable module category are
defined by common properties with the set of simple modules, whose
images under stable equivalences do form simple-minded systems. Over
a representation-finite self-injective algebra, it is shown that all
simple-minded systems are images of simple modules under stable
equivalences of Morita type, and that all simple-minded systems can
be lifted to Nakayama-stable simple-minded
collections in the derived category. In particular, all simple-minded
systems can be obtained algorithmically using mutations.
\end{abstract}

\maketitle

\section{Introduction}
Module categories contain two kinds of especially important objects:
From {\it simple} modules other objects can be produced by
iteratively forming extensions. From {\it projective} modules other
objects can be produced by considering presentations or resolutions.
Moreover, by Morita theory, projective objects control equivalences
of module categories. The role of projective modules can in derived
categories be taken over by appropriate generalisations
(``projective-minded'' objects satisfying certain homological
conditions) such as tilting complexes, which still control
equivalences of such categories. In stable categories, no
substitutes of projective objects are known and stable equivalences
are, in general, not known to be controlled by particular objects.
It is not even known whether equivalences of stable module
categories of finite dimensional algebras preserve the number of
non-projective simple modules (up to isomorphism); the Auslander-Reiten conjecture -- which appears to be wide open -- predicts a
positive answer to this question.

The images of simple modules under a stable equivalence do keep some
of the properties of simple objects such as their endomorphism ring
being a skew-field and every non-zero homomorphism between them
being an isomorphism. Moreover, they still generate the stable
category and there is no cohomology between them in negative
degrees. Such systems of objects in a stable module category have
been called  simple-minded systems in \cite{KL}. Analogous systems
of objects in a derived module category (defined in a slightly
different way) have been called cohomologically Schurian collections
in \cite{Al-Nofayee2007} and simple-minded collections in \cite{KY}.

Any information on simple-minded systems for an algebra can help to
describe the still rather mysterious stable module category and in
particular equivalences between stable categories. The following two
problems appear to be crucial:

\smallskip
{\it The simple-image problem:} Is every simple-minded system the
image of the set of simples of some algebra under some stable
equivalence?

\smallskip
{\it The liftability problem:} Is there a connection between the
simple-minded systems in the stable category of a self-injective
algebra and the simple-minded collections in its derived module
category? More precisely, are the simple-minded systems images of
simple-minded collections under the quotient functor from the
derived to the stable category?

Note that when the algebra is self-injective, its stable module
category is a quotient of its derived module category.

On a numerical level, a positive answer to the question if all
simple-minded systems of an algebra have the same cardinality implies validity of the
Auslander-Reiten conjecture. The information we are looking for is
stronger and is part of an attempt to better understand the
structure of stable categories and stable equivalences.
\smallskip

Expecting positive answers to these questions appears to be rather
optimistic. In this article we do, however, provide positive answers
to both problems for the class of representation-finite
self-injective algebras, which includes for instance all the blocks
of cyclic defect of group algebras of finite groups over fields of
arbitrary characteristic.

\smallskip
{\bf Theorem A (\ref{orbits-sms} and \ref{sms-smc}):} {\it Let $A$
be self-injective. Then there are injective maps
\\
\[
\begin{array}{c}
\xymatrix@1@-10pt{
StM Alg(A)  \ar@{^{(}->}[dr]  & &   \ar@{^{(}->}[dl] smc(A) / DPic(A) \\
& sms(A)/StPic(A) & }
\end{array}
\]
\\
which send an algebra, which is stably equivalent of Morita type
to $A$, to a simple-minded system, and a
{Nakayama-stable} simple-minded collection also to a
simple-minded system. The left hand map is a bijection if and only
if every simple-minded system is the image of simples under a stable
equivalence of Morita type. The right hand map is a bijection if and
only if every simple-minded system is the image of simples under a
stable equivalence that lifts to a derived equivalence.}
\medskip

For representation-finite self-injective algebras, the criteria are
satisfied. This is the main result of this article.
\smallskip

{\bf Theorem B (\ref{stronger-transsms} and \ref{RFS-sms-smc-bijec}:)} {\it
Let $A$ be self-injective of finite representation type over an
algebraically closed field. Then the two maps in Theorem A are
bijections.}
\medskip

\medskip

A main tool for proving Theorem B is a combinatorial description of
simple-minded systems:
\smallskip
{\bf Theorem C (\ref{sms-configuration}):} {\it Let $A$ be
self-injective of finite representation type over an algebraically
closed field. Then there is a bijection between simple-minded
systems and Riedtmann's configurations.}

\medskip
{Note that all stable equivalences in this situation
turn out to be liftable stable equivalences of Morita type, that is, they can
be lifted to standard derived equivalences.} Analysing this situation
in detail also yields an unexpected property of
simple-minded systems in this case; they are all Nakayama-stable.
This stability appears to be a crucial property that is potentially
useful in other situations, too. Adopting this point of view allows
us to simplify a proof of Dugas \cite{Dugas2011} on the
liftability of stable equivalences between particular
representation-finite self-injective algebras.

Simple-minded systems may be compared with other concepts that arise
for instance in cluster theory or in the emerging generalisation of
tilting to silting. These concepts also come with a theory of mutation.
Therefore, it makes sense to ask for the phenomena which replicate in different
situations. In this context, we will prove the following result, that is
formally independent of simple-minded systems, but intrinsically
related to our approach:

\smallskip
{\bf Theorem D (\ref{RFS-tilting-connected}):} {\it Let $A$ be a
self-injective algebra of finite representation type over an algebraically
closed field. Then the homotopy
catgory $K^b(proj A)$  is strongly tilting connected.}
\medskip

Combining this with other results, we show an analogous result
for the stable module category. In particular, we get
that all simple-minded systems in this case can be obtained by
iterative left irreducible mutations starting from simple modules
(see \ref{mutation-transitive-sms}).
\medskip

The proofs use a variety of rather strong results and methods from
the literature, including covering theory, Riedtmann's description
of configurations of representation finite self-injective algebras,
Asashiba's classification results on stable and derived
equivalences, Asashiba's and Dugas' results on liftability of stable
equivalences, and various mutation theories.
\medskip

This article is organised as follows. Section 2 contains some
general statements on sms's over self-injective algebras: their
connection with smc's; the relationship between the orbits of sms's
under stable Picard group and the Morita equivalence classes of
stably equivalent algebras. We shall formulate the basic problems
about sms: simple-image problem and liftability of simple-image sms
problem. The main result of this section is Theorem A, which is
valid for self-injective algebras in general. It determines the point
of view taken in this article.

From Section 3, we restrict our discussion to representation-finite
self-injective algebras over an algebraically closed field. Section
3 gives the correspondence between configurations and sms's, that
is, Theorem C. We deal with both the standard and the non-standard
case. As a consequence, we can solve the simple-image problem of
sms's for representation-finite self-injective algebras.

Section 4 gives the correspondence between orbits of sms's of the
stable category and orbits of Nakayama-stable smc's of the bounded
derived category, that is, one assertion in Theorem B. This is based
on a lifting theorem for stable equivalences between
representation-finite self-injective algebras. This lifting theorem
also allows us to give positive answer to a stronger form of
simple-image problem, which completes Theorem B.

In Section 5 we discuss some aspects of the various mutations of
different objects: tilting complex, smc, and sms. We will show that
the sms's of a representation-finite self-injective algebra can be
obtained by iterative mutations. As a by-product of our point of
view we obtain Theorem D.
\bigskip

\section{Statement of problems, and their motivations}
Let $k$ be a field and $A$ a finite dimensional self-injective $k$-algebra.
\medskip

We denote by mod$A$ the category of all finitely generated left
$A$-modules, by mod$_\mathcal{P}A$ the full subcategory of mod$A$
whose objects have no nonzero projective direct summand, and by
$\stmod A$ the stable category of mod$A$ modulo projective modules.
Let $\mathcal{S}$ be a class of A-modules. The full subcategory
$\langle\mathcal{S}\rangle$ of mod$A$ is the additive closure of
$\mathcal{S}$. Denote by
$\langle\mathcal{S}\rangle\ast\langle\mathcal{S'}\rangle$ the class
of indecomposable $A$-modules $Y$ such that there is a short exact
sequence $0 \rightarrow X \rightarrow Y\oplus P \rightarrow Z\rightarrow 0$ with $X\in \langle\mathcal{S}\rangle, Z\in\langle\mathcal{S'}\rangle$,
and $P$ projective. Define $\langle\mathcal{S}\rangle_1:=\langle\mathcal{S}\rangle$ and
$\langle\mathcal{S}\rangle_n:=\langle\mathcal{S}\rangle_{n-1}\ast\langle\mathcal{S}\rangle$ for $n>1$.

To study sms's over $A$, without loss of generality, we may assume
the following throughout the article: $A$ is indecomposable
non-simple and contains no nodes (see \cite{KL}). We can then
simplify the definition of sms from \cite{KL} as follows.

\begin{Def}\label{sms} {\rm(\cite{KL})}
Let $A$ be as above.  A class of objects $\mathcal{S}$ in
mod$_\mathcal{P}A$ is called a simple-minded system (sms) over $A$
if the following conditions are satisfied:
\begin{enumerate}
\item (orthogonality condition) For any $S,T\in\mathcal{S}$,
$\underline{Hom}_A(S,T)= \left\{\begin{array}{ll} 0 & (S\neq T), \\
\text{division ring} & (S=T).\end{array}\right.$
\item (generating condition) For each indecomposable non-projective
$A$-module $X$, there exists some natural number $n$ (depending on
$X$) such that $X\in\langle\mathcal{S}\rangle_n$.
\end{enumerate}
\end{Def}

It has been shown in \cite{KL} that each sms {has} finite
cardinality and the sms's are invariant under stable equivalence,
{i.e. the image of an sms under a stable equivalence is also an
sms.}  Note that the set of simple $A$-modules clearly forms an sms.
We are going to present two fundamental problems, as noted in the
introduction, on the study of sms, and we provide motivations for
them. The first one is the \textit{simple-image problem}:
\begin{Prob}\label{simple-image-problem}
Simple-image problem:
\begin{enumerate}
\item Given an sms $\mathcal{S}$ of $A$, is this the image of the
simple modules under a stable equivalence?  (When this is true, we
say $\mathcal{S}$ is a simple-image sms, or shorter, it is simple-image.)
\item Is every sms of $A$ simple-image?
\end{enumerate}
\end{Prob}
For some technical reasons, we will usually consider a stronger
version of this problem, where we replace stable equivalence by
\textit{stable equivalence of Morita type} (see the definition
below).  When $\mathcal{S}$ is the image of simple modules under a
stable equivalence of Morita type, we say $\mathcal{S}$ is a
\textit{simple-image sms of Morita type}.  The strong version of (2)
is ``Is every sms of $A$ a simple-image of Morita type?".  Our aim
is to solve the strong simple-image problem in the case of
representation-finite self-injective algebras over
algebraically closed fields.

In \cite{KL}, a weaker version of sms has been introduced, and it
has been shown that when $A$ is representation-finite
self-injective, the following system is sufficient (hence
equivalent) for defining an sms.
\begin{Def}
\label{wsms} {\rm(\cite{KL})} Let $A$ be as in Definition \ref{sms}.
A class of objects $\mathcal{S}$ in mod$_\mathcal{P}A$ is called a
weakly simple-minded system (wsms) if the following two conditions
are satisfied:
\begin{enumerate}
\item (orthogonality condition) For any $S,T\in\mathcal{S}$,
$\underline{Hom}_A(S,T)= \left\{\begin{array}{ll} 0 & (S\neq T), \\
\text{division ring} & (S=T).\end{array}\right.$
\item (weak generating condition) For any indecomposable non-projective
$A$-module $X$, there exists some $S\in \mathcal{S}$ (depends on
$X$) such that $\underline{Hom}_A(X,S)\neq 0.$
\end{enumerate}
\end{Def}

A similar concept used for derived module categories is the
simple-minded collection (smc) of \cite{KY}, which coincides with
the cohomologically Schurian collection of Al-Nofayee
\cite{Al-Nofayee2007}.
\begin{Def}\label{smc} {\rm(\cite{KY})}
A collection $X_1,\cdots,X_r$ of objects in a triangulated category
$\mathcal{T}$ is simple-minded if for $i,j=1,\cdots,r$, the
following conditions are satisfied:
\begin{enumerate}
\item (orthogonality) $Hom(X_i, X_j) = \begin{cases} \text{division ring} & \mbox{if }i=j,\\ 0 &\mbox{otherwise};\end{cases}$
\item (generating) $\mathcal{T} = thick(X_1\oplus\cdots\oplus X_r)$;
\item (silting/tilting) $Hom(X_i,X_j[m])=0$ for any $m<0$.
\end{enumerate}
\end{Def}
For any (finite dimensional) $k$-algebra $A$, the simple $A$-modules
form a simple-minded collection of the bounded derived category
$D^b$(mod$A$). Simple-minded collections appeared already in the
work of Rickard \cite{Rickard2002}, who constructed tilting
complexes inducing equivalences of derived categories that send a
simple-minded collection for a symmetric algebra to the simple
modules of another symmetric algebra. Al-Nofayee
\cite{Al-Nofayee2007} generalised Rickard's work to self-injective
algebras, requiring an smc to satisfy the following
Nakayama-stability condition. Recall that for a self-injective
algebra $A$, the Nakayama functor $\nu_A=Hom_k(A,k)\otimes_A-:$
mod$A\rightarrow$ mod$A$ is an exact self-equivalence and therefore
induces a self-equivalence of $D^b$(mod$A$) which will also be
denoted by $\nu_A$. By Rickard \cite{Rickard1991}, if $\phi:D^b$(mod$A)\rightarrow D^b$(mod$B$) is a derived equivalence between
two self-injective algebras $A$ and $B$, then $\phi\nu_A(X)\simeq\nu_B\phi(X)$ for any object $X\in D^b$(mod$A$). We shall say an smc
$X_1,\cdots,X_r$ of $D^b$(mod$A$) is {\it Nakayama-stable} if the
Nakayama functor $\nu_A$ permutes $X_1,\cdots,X_r$. In particular, any derived equivalence $\phi:
D^b$(mod$A)\rightarrow D^b$(mod$B$) sends simple modules to a Nakayama-stable smc.
\medskip

Let $A$ and $B$ be two algebras. Following Brou\'{e}
\cite{Broue1994}, we say that there is a {\it stable equivalence of
Morita type} (StM) $\phi: \stmod A\rightarrow \stmod B$  if there
are two left-right projective bimodules $_AM_B$ and $_BN_A$ such
that the following two conditions are satisfied:
\begin{enumerate}
\item ${}_A M\otimes_B N_A\simeq {}_AA_A\oplus {}_AP_A,\ \ \ {}_B N\otimes_A M_B\simeq {}_BB_B\oplus {}_BQ_B,$ \\ where ${}_AP_A$ and
${}_BQ_B$ are some projective bimodules;
\item $\phi$ is a stable equivalence which lifts to the functor $N\otimes_{A}-$, that is, the diagram
$$\xymatrix{modA\ar[r]^{N\otimes_{A}-}\ar[d]_{\pi_A}& modB\ar[d]^{\pi_B}\\
\stmod A\ar[r]^{\phi}& \stmod B}$$
commutes up to natural isomorphism, where $\pi_A$ and $\pi_B$ are
the natural quotient functors.
\end{enumerate}
This special class of stable equivalences occurs frequently in
representation theory of finite groups, and more generally, in
representation theory of finite dimensional algebras (see, for
example, \cite{Broue1994}, \cite{Rickard1991},
\cite{Linckelmann1996}, \cite{Linckelmann1998}, \cite{LX}). We will
frequently use the following two well-known results of Rickard and
Linckelmann.  The former says that for a self-injective $A$, the
embedding functor mod$A\rightarrow D^b$(mod$A$) induces an
equivalence $\stmod A\rightarrow D^b$(mod$A$)/$K^b($proj$A)$. So
there is a natural quotient functor $\eta_A:D^b$(mod$A)\rightarrow\stmod A$ of triangulated categories. A standard derived equivalence
between two self-injective algebras induces a StM (here a {\it
standard derived equivalence} means that it is isomorphic to the
functor given by tensoring with a two-sided tilting complex, see
\cite{Rickard1989,Rickard1991,Asashiba2003} for more details).
Linckelmann \cite{Linckelmann1996} showed that a StM between two
self-injective algebras lifts to a Morita equivalence if and only if
it sends simple modules to simple modules.

The second fundamental problem asks how a simple-image sms of Morita
type is related to Nakayama-stable smc:
\begin{Prob}\label{liftable-sms}
The liftability problem for simple-image sms's of Morita type is:
Given a simple-image sms $\mathcal{S}$ of Morita type under StM,
$\phi:\stmod B\to \stmod A$ with $B$ self-injective, can $\phi$ be
lifted to a derived equivalence?
\end{Prob}

Given a simple-image sms $\mathcal{S}$ under a liftable StM $\phi$
as above, we simply say $\mathcal{S}$ is a \textit{liftable simple-image} sms.  We will justify our terminology in Proposition \ref{orbits-sms}.

Next we recall the notion of stable Picard group from
\cite{Linckelmann1998,Asashiba2003}. Let $A$ be an algebra. The
{more conventional notion of} {\it Picard group} Pic$(A)$ of $A$ is
defined to be the set of natural isomorphism classes of Morita
self-equivalences over $A$. The set StPic$(A)$ of natural
isomorphism classes $[\phi]$ of StM $\phi: \stmod A\rightarrow\stmod A$ form a group under the composition of functors, which is
called the {\it stable Picard group} of $A$. Notice that the
definitions for stable Picard group used by Linckelmann
\cite{Linckelmann1998} and by Asashiba \cite{Asashiba2003} are
different even in the case of representation-finite self-injective
algebras. Linckelmann used the isomorphism classes of bimodules
which define StM, while Asashiba used the isomorphism classes of
\textit{all} stable self-equivalences. {We use {the one closer to}
Linckelmann's version of stable Picard group in the propositions to
follow. In Section 4 we will specify the link between the two
versions when $A$ is representation-finite.} Similarly we define the
{\it derived Picard group} DPic$(A)$ of $A$ as the set of natural
isomorphism classes of standard derived self-equivalences of the
bounded derived category $D^b$(mod$A$). Clearly each Morita
equivalence: mod$A\rightarrow$ mod$A$ induces a StM: $\stmod A\rightarrow \stmod A$. We denote the image of the canonical
homomorphism Pic$(A)\rightarrow$ StPic$(A)$ by Pic$'(A)$.  Note that two
non-isomorphic bimodules may induce isomorphic StM, which is the
reason why we use Pic$'(A$) here.  This
distinction will become important in Section 4.

Let $A$ be an algebra. In the following, we will identify two sms's
of $A$, $\mathcal{S}_1=\{X_1,\cdots,X_r\}$ and
$\mathcal{S}_2=\{X_1',\cdots,X_s'\}$, if $r=s$ and $X_i\simeq X_i'$
for all $1\leq i\leq r$ up to a permutation.  We use the same
convention for smc's.  {We use calligraphic font (e.g.
$\mathcal{S}$) and bold font (e.g. $\mathbf{S}$) for sms's and smc's
respectively to distinguish the two.}  Now we fix some notations:

$\mathcal{S}_A=\{\mbox{isomorphism classes of simple}A\mbox{-modules}\}$;

$\mathrm{StMAlg}(A)=\{\mbox{the Morita equivalence classes of algebras which are StM to } A\};$

$\mathrm{sms}(A)/\mathrm{StPic}(A)=\{\mbox{the orbits of sms's of
}\stmod A\mbox{ under StPic}(A)\};$

$\mathrm{smc}(A)/\mathrm{DPic}(A)=\{\mbox{the orbits of Nakayama-stable smc's of }D^b(\mathrm{mod}A)\mbox{ under
DPic}(A)\}.$

\begin{Prop}\label{unique-algebra}
Let $A$ be a self-injective algebra.  Strong simple-image problem
(see \ref{simple-image-problem}) has a positive answer for $A$, if
and only if, every sms $\mathcal{S}$ of $A$ is simple-image under a
StM $\phi:\stmod B\to \stmod A$, where the algebra $B$ is uniquely
determined by $\mathcal{S}$, up to Morita equivalence.
\end{Prop}

\begin{Proof}
A positive answer to the strong simple-image problem asserts that
every sms $\mathcal{S}$ of $A$ is of the form $\phi(\mathcal{S}_B)$,
where $\phi:\stmod B\to \stmod A$ is a StM.  Suppose there is
another stably equivalent algebra $B'$ with StM $\phi ' :\stmod B'\to \stmod A$ and $\phi '(\mathcal{S}_{B'})=\mathcal{S}$,
then $\phi^{-1}\phi '$ sends $\mathcal{S}_{B'}$ to $\mathcal{S}_B$.  By
Linckelmann's Theorem, $\phi^{-1}\phi '$ is then a Morita
equivalence.  The other direction is clear.
\end{Proof}

\begin{Rem}
(1)  This is true for arbitrary finite dimensional algebras when we
replace ``simple $B$-modules" by ``non-projective simple
$B$-modules", due to Linckelmann's theorem being valid for general
finite dimensional algebras (see \cite{Liu2003} and \cite[Section 4]{KY}).

(2)  Uniqueness is false if we relax the right hand side statement
by arbitrary stable equivalence, even in the representation-finite
case. For example, when $A$ is Nakayama algebra with two {simples}
and Loewy length two, then {$\mathcal{S}_A$ clearly is a
simple-image sms of Morita type. However, $\mathcal{S}_A$ is also the image
of the simple $B$-modules}, where
$B = k[x]/(x^2)\times k[y]/(y^2)$, under a $k$-linear (non-triangulated) equivalence of stable module
categories.
\end{Rem}

\begin{Thm} \label{orbits-sms} Let $A$ be a self-injective algebra.
Let $\mathrm{StMAlg}(A)$ and $\mathrm{sms}(A)/\mathrm{StPic}(A)$ be
as above. Then:
\begin{enumerate}
\item There is a well-defined map from $\mathrm{StMAlg}(A)$ to
$\mathrm{sms}(A)/\mathrm{StPic}(A)$.
\item This map is injective.  It is a bijection if and only if every
sms of $A$ is simple-image of Morita type.
\end{enumerate}
\end{Thm}
\begin{Proof}
(1) For an algebra $B$, denote by $[B]$ the Morita equivalence class
of $B$. Let $[B]\in \mathrm{StMAlg}(A)$ and fix a StM $\phi: \stmod B\rightarrow \stmod A$. Then the image $\phi(\mathcal{S}_B)$ of the
simple $B$-modules is an sms over $A$.  We denote by
$[\phi(\mathcal{S}_B)]\in \mathrm{sms}(A)/\mathrm{StPic}(A)$ the
orbit of $\phi(\mathcal{S}_A)$ under the stable Picard group
StPic$(A)$. Let $B'$ be another algebra with $\sigma:\mathrm{mod}B'\rightarrow \mathrm{mod}B$ a Morita equivalence. Let
$\psi: \stmod B'\rightarrow \stmod A$ be a StM with
$\psi(\mathcal{S}_{B'})$ the image of the simple $B'$-modules. Then
$\psi\sigma\phi^{-1}\in$ StPic$(A)$ maps $\phi(\mathcal{S}_B)$ onto
$\psi(\mathcal{S}_B')$ and therefore
$[\phi(\mathcal{S}_B)]=[\psi(\mathcal{S}_{B'})]$ in
sms$(A)$/StPic($A$), showing the existence of a well-defined map
from $\mathrm{StMAlg}(A)$ to sms$(A)$/StPic($A$).

(2) To show this map is injective, we consider two pairs $(B,\phi)$
and $(B',\phi')$, where $B$ and $B'$ are two algebras such that
$[B],[B']\in \mathrm{StMAlg}(A)$ and where $\phi: \stmod B\rightarrow \stmod A$ and $\phi ': \stmod B'\rightarrow \stmod A$
are two StM. We have two sms's $\phi(\mathcal{S}_B)$ and $\phi'(\mathcal{S}_{B'})$ over $A$.  Suppose that $\phi(\mathcal{S}_B)$
is mapped to $\phi '(\mathcal{S}_{B'})$ by an element $\rho\in$
StPic$(A)$.  Then ${\phi '}^{-1}\rho\phi:\stmod B\rightarrow \stmod B'$ is a StM which maps $\mathcal{S}_B$ to $\mathcal{S}_{B'}$.  By
Linckelmann's theorem, $\phi^{-1}\rho\phi$ is a Morita equivalence,
and so $B$ and $B'$ are Morita equivalent. This finishes the proof
of the injectivity.

Finally, {the previous proposition asserts there is a well-defined
inverse map if and only if every sms of $A$ is simple-image of
Morita type}.
\end{Proof}
\begin{Rem}
(1) We will see in Section 4 that the above map is a bijection in
case that $A$ is a  representation-finite self-injective algebra.

(2) We do not know whether there is an example with a non-bijective
map.  Note that the algebra $A$ in Example 3.5 of \cite{KL} is in
fact not a counterexample to the strong simple-image problem
(despite a misleading formulation in \cite{KL}): there is a  {StM
from $A$ to the following Brauer tree algebra $B$} such that the sms
$\mathcal{S}_2$ over $A$ is mapped to simple $B$-modules:
\unitlength=1.00mm \special{em:linewidth 0.4pt}
\linethickness{0.4pt}
\begin{center}
\begin{picture}(12.00,12.00)
\put(-9,3){$B$} \put(-5,3){$=$}\put(0,7){$1$} \put(0,3){$2$}
\put(0,-1){$1$} \put(4,3){$\oplus$} \put(12,7){$2$} \put(8,3){$1$}
\put(16,3){$2$}\put(12,-1){$2$}
\end{picture}
\end{center}
(3) Similar to the previous proposition, this proposition is true
also for any finite dimensional algebra once we replace the simple
modules by non-projective simple modules in the argument.
\end{Rem}

\begin{Thm} \label{sms-smc} Let $A$ be a self-injective algebra.
Let smc($A$)/DPic($A$) and sms$(A)$/StPic($A$) be as above. Then:
\begin{enumerate}
\item Every Nakayama-stable smc of $D^b$(mod$A$) determines an sms
of $\stmod A$ under the natural functor
$\eta_A:D^b$(mod$A)\rightarrow \stmod A$. Conversely, an sms
$\mathcal{S}$ of $\stmod A$ lifts to a Nakayama-stable smc of
$D^b$(mod$A$) if $\mathcal{S}$ is a liftable simple-image sms (see
Problem \ref{liftable-sms}).
\item There is an injective map from smc$(A)$/DPic(A) to
sms$(A)$/StPic($A$). This map is a bijection if and only if
every sms $\mathcal{S}$ of $A$ is a liftable simple-image.
\end{enumerate}
\end{Thm}
\begin{Proof}
(1) Let $\mathbf{S}=\{X_1,\cdots,X_r\}$ be a Nakayama-stable smc of
$D^b$(mod$A$). By Al-Nofayee \cite{Al-Nofayee2007}, there exists a
self-injective algebra $B$ (unique up to Morita equivalence) and a
derived equivalence $\phi: D^b$(mod$B)\rightarrow D^b$(mod$A$) such
that $\phi$ sends simple $B$-modules onto $\mathbf{S}$. By Rickard
\cite[Corollary 3.5]{Rickard1991}, we can assume that $\phi$ is a
standard derived equivalence. Notice that the number $r$ must be
equal to the number of (isoclasses of) simple $A$-modules, since a
derived equivalence preserves the Grothendieck group. By Rickard
\cite{Rickard1989,Rickard1991}, $\phi$ induces a StM
$\overline{\phi}:\stmod B\rightarrow \stmod A$ so that the following
commutative diagram

$$\xymatrix{D^b(modB)\ar[r]^{\phi}\ar[d]_{\eta_B}& D^b(modA)\ar[d]^{\eta_A}\\\stmod B\ar[r]^{\overline{\phi}}& \stmod A}$$
commutes up to natural isomorphism. Since $\eta_B$ is the identity
on modules, $\overline{\phi}$ sends simple $B$-modules onto
$\eta_A(\mathbf{S}_1)=\{\eta_A(X_1),\cdots,\eta_A(X_r)\}$, and
therefore $\eta_A(\mathbf{S}_1)$ is an sms over $A$.

Conversely, suppose that $\mathcal{S}$ is a {liftable simple-image
sms}. Then there is a StM $\overline{\phi}:\stmod B\rightarrow\stmod A$ such that $\overline{\phi}$ sends simple $B$-modules onto$\mathcal{S}$ and
that $\overline{\phi}$ lifts to a derived
equivalence $\phi: D^b$(mod$B)\rightarrow D^b$(mod$A$). Again by
Rickard \cite[Corollary 3.5]{Rickard1991}, we can assume that $\phi$
is a standard derived equivalence. It follows that the image
$\mathbf{S}$ of simple $B$-modules under $\phi$ is a Nakayama-stable
smc of $D^b$(mod$A$), which is clearly a lifting of the sms
$\mathcal{S}$.

(2) Let $\mathbf{S}$ be a Nakayama-stable smc of $D^b$(mod$A$) as in
(1). Suppose that $\mathcal{S}'=\{X_1',\cdots,X_r'\}$ is another
Nakayama-stable smc of $D^b$(mod$A$). As before, there is another
algebra $B'$ and a standard derived equivalence $\phi
':D^b$(mod$B')\rightarrow D^b$(mod$A$) such that $\phi '$ sends
simple $B'$-modules onto $\mathbf{S}'$. Similarly, the induced StM
$\overline{\phi'}:\stmod B'\rightarrow\stmod A$ sends simple
$B'$-modules onto
$\eta_A(\mathbf{S}')=\{\eta_A(X_1'),\cdots,\eta_A(X_r')\}$, and
$\eta_A(\mathbf{S}')$ is an sms over $A$. Suppose moreover, that
$[\eta_A(\mathbf{S})]=[\eta_A(\mathbf{S}')]$ in sms$(A)$/StPic($A$),
that is, there is a StM $\alpha:\stmod A\rightarrow \stmod A$
sending the sms $\eta_A(\mathbf{S})$ to $\eta_A(\mathbf{S}')$. Then
the composition $\overline{\phi'}^{~-1}\alpha\overline{\phi}:\stmod
B\rightarrow \stmod B'$ sends simple modules to simple modules, and
by Linckelmann's theorem, $\overline{\phi
'}^{~-1}\alpha\overline{\phi}$ is lifted to a Morita equivalence:
mod$B\rightarrow$ mod$B'$, which again induces a derived equivalence
$\beta: D^b$(mod$B)\rightarrow D^b$(mod$B'$) sending simple
$B$-modules to simple $B'$-modules. The composition $\phi
'\beta\phi^{-1}:D^b$(mod$A)\rightarrow D^b$(mod$A$) of the derived
equivalences sends the Nakayama-stable smc $\mathbf{S}$ onto
$\mathbf{S}'$, and so we have $[\mathbf{S}]=[\mathbf{S}']$ in
smc$(A)$/DPic($A$). We have shown that there is an injective map
from smc$(A)$/DPic($A$) to sms$(A)$/StPic($A$). It is not difficult
to see that this map is a bijection if and only if {every sms
$\mathcal{S}$ of $A$ is a liftable simple-image.}
\end{Proof}
\begin{Rem}
{(1) This proposition justifies the terminology ``liftable
simple-image sms", in which case the sms considered can then be
lifted to an (Nakayama-stable) smc.}

(2) We will see in Section 4 that the above map is a bijection in
case that $A$ is a representation-finite self-injective
algebra.

{(3) In the representation-infinite case, there exists simple-image
sms $\mathcal{S}$ of Morita type under a non-liftable StM. For
example, let $A$ and $B$ be the principal} blocks of the Suzuki
group $S_z(8)$ and of the normalizer of a Sylow $2$-subgroup of
$S_z(8)$ over a field $k$ of characteristic $2$. Then $A$ and $B$
are stably equivalent of Morita type, say under $\phi$, but not
derived equivalent by \cite{Broue1994}. Obviously,
$\mathcal{S}=\phi(\mathcal{S}_B)$ is a simple-image sms of Morita
type over $A$.  If there is another algebra $C$ so that $\psi:\stmod C\to \stmod A$ is a stable equivalence sending $\mathcal{S}_C$ to
$\mathcal{S}$ and $\psi$ liftable, then
$\phi^{-1}\psi(\mathcal{S}_C)=\mathcal{S}_B$.  By Linckelmann's
theorem $C$ and $B$ are Morita equivalent, as $A$ and $C$ are
derived equivalent. This implies that $A$ and $B$ also are derived
equivalent, which is a contradiction.  Therefore, we have an example
of a simple-image sms of Morita type which is \textit{never} liftable.
\end{Rem}
\bigskip

\section{Sms's and configurations}
In this section we are going to answer the simple-image problem by
proving Theorem C. Following Asashiba \cite{Asashiba1999}, we
abbreviate (indecomposable, basic) representation-finite
self-injective algebra (not isomorphic to the underlying field $k$)
by RFS algebra.

\begin{Thm}
\label{transsms} Let $A$ be an RFS algebra over an
algebraically closed field, and $\mathcal{S}$ an {sms} of $A$. Then
there is an RFS algebra $B$ and a stable equivalence from
$\stmod B$ to $\stmod A$ such that the set of simple $B$-modules is
mapped to $\mathcal{S}$ under the stable equivalence, i.e.
$\mathcal{S}$ is a simple-image sms.
\end{Thm}
{\it Strategy of proof.} In Theorem \ref{sms-configuration} below
the simple-minded system $\mathcal{S}$ corresponds to a
configuration $\mathcal{C}$ in the stable AR-quiver $_s\Gamma_A$.
Configurations correspond to RFS algebras. It follows that there is
an RFS algebra $B$ and a stable equivalence $\phi:\stmod A\rightarrow\stmod B$ such that $\phi(\mathcal{S})$ is precisely the set
\{$rad(P)|P$ an (isoclass of) indecomposable projective
$B$-module\}. Applying the Heller operator $\Omega_B$ we get a
stable equivalence $\phi^{-1}\Omega_B:\stmod B\rightarrow \stmod A$
sending simple $B$-modules onto $\mathcal{S}$. Note that this proof
is constructive.
\smallskip

\begin{Rem}\label{rmk-RiedtmannUniqueAlg}
\begin{enumerate}
\item We will see in Section 4 that, for an RFS algebra $A$,
all sms's of $A$ are in fact simple-image of Morita type.
\item The classification theorem of RFS algebras, first proved in the 80's,
does already imply implicitly that $B$ is determined uniquely up to
Morita equivalence.
\end{enumerate}
\end{Rem}

\subsection{Proof of Theorem \ref{transsms}}
Now we start working out the details of the above sketch of proof.
The main tools come from Riedtmann's work on RFS algebras and their AR-quivers, and from Asashiba's work on stable and derived equivalences
between RFS algebras. We use standard definitions of AR theory without
explanations; see \cite{ARS, ASS, BG} for details. In the following we
recall the definitions of configurations and combinatorial
configurations, and see how these notions are translated into the
setting of sms's. Throughout this section $Q$ denotes a Dynkin
quiver of type $A_n, D_n, E_6, E_7$ or $E_8$; and $\mathbb{Z}Q$ is
the corresponding translation quiver with translation denoted as
$\tau$. For a translation quiver $\Gamma$, we let $k(\Gamma)$ be its
mesh category, that is, the path category whose objects are the vertices of
$\Gamma$; morphisms are generated by arrows of $\Gamma$ quotiented
out by the mesh relations. Riedtmann showed in \cite{Riedtmann1}
that for an RFS algebra over an algebraically closed field, the
stable AR-quiver is of the form $\mathbb{Z}Q/\Pi$ for some
admissible group $\Pi$. Consequently we say such algebra is of tree
class $Q$ and has admissible group $\Pi$. Note that we always assume the RFS
algebras considered to be indecomposable, basic and not isomorphic to the
underlying field $k$.

\begin{Def}\label{configuration}{\rm(\cite{BLR})}
A configuration of $\mathbb{Z}Q$ is a subset $\mathcal{C}$ of
vertices of $\mathbb{Z}Q$ such that the quiver
$\mathbb{Z}Q_\mathcal{C}$ is a representable translation quiver.
$\mathbb{Z}Q_\mathcal{C}$ is constructed by adding one vertex $c^*$
for each $c\in\mathcal{C}$ on $\mathbb{Z}Q$; adding arrows
$c\rightarrow c^*\rightarrow \tau^{-1}c$; and letting the
translation of $c^*$ be undefined.
\end{Def}

Here, the following notation is used: A translation quiver is
representable if and only if its mesh category is an Auslander
category. We do not go through the technicalities of these
definitions; the reader can bear in mind that the mesh category of
the Auslander-Reiten quiver (or its universal cover) of an representation-finite algebra is
an Auslander category (see \cite{BG}). The idea is that for
$\Pi$-stable configuration $\mathcal{C}$, $\mathbb{Z}Q/\Pi$ is the
stable AR-quiver of an RFS algebra and $\mathbb{Z}Q_\mathcal{C}/\Pi$
is the AR-quiver of the algebra, where the extra (projective)
vertices $c^*$ are the vertices representing the (isoclasses of)
indecomposable projective modules of the algebra. In particular, the
set $\{rad(P)|P\mbox{ an (isoclass of) indecomposable projective}\}$
of an RFS algebra is a configuration.

\begin{Def}\label{combinatorial-configuration} {\rm(\cite{Riedtmann2})}
Let $\Delta$ be a stable representable quiver. A combinatorial
configuration $\mathcal{C}$ is a set of vertices of $\Delta$ which
satisfy the following conditions:
\begin{enumerate}
\item For any $e, f\in \mathcal{C}$,
$Hom_{k(\Delta)}(e,f)= \left\{\begin{array}{ll} 0 & (e\neq f), \\
k & (e=f).\end{array}\right.$
\item For any $e\in\mathbb{Z}Q$, there exists some $f\in \mathcal{C}$ such that $Hom_{k(\Delta)}(e,f)\neq 0$.
\end{enumerate}
\end{Def}

We also note the following fact in \cite[Proposition 2.3]{Riedtmann2}: if $\pi: \Delta\rightarrow \Gamma$ is a covering,
then $\mathcal{C}$ is a combinatorial configuration of $\Gamma$ if
and only if $\pi^{-1}\mathcal{C}$ is a combinatorial configuration
of $\Delta$. When applied to the universal cover of stable AR-quiver
of RFS algebra $A$, this translates to the following statement: The
$\Pi$-stable configuration of the universal cover is a configuration
of the stable AR-quiver $\mathbb{Z}Q/\Pi$.

Combinatorial configurations have been defined by Riedtmann when
studying self-injective algebras \cite{Riedtmann2}. At first this is
a generalisation of configuration. It is often easier to study and
compute than a configuration as it suffices to look
`combinatorially' at sectional paths of the translation quiver
$\mathbb{Z}Q$ rather than checking whether $k(\Z Q_{\mathcal{C}})$
can be realised as an Auslander category. Therefore, it is
interesting to know if these two concepts coincide. In the case of
RFS algebras, this is true. As mentioned in the sketch previously, a
configuration represents a set
 \{$rad(P)|P$ an (isoclass of) indecomposable projective\}. Applying the
inverse Heller operator $\Omega^{-1}$, which is an auto-equivalence
of the stable category of an RFS algebra, the above set is mapped to
the set of simples of the RFS algebra. Indeed, in
\cite{Riedtmann2,Riedtmann3,BLR} it has been shown that $\Pi$-stable
configuration of $\mathbb{Z}Q$ and combinatorial configuration of
$\mathbb{Z}Q/\Pi$ do coincide. Thus in the following, for an RFS
algebra $A$, we can identify the configurations and combinatorial
configurations of the stable AR-quiver $_s\Gamma_A$.

In \cite{Riedtmann2,Riedtmann3,BLR}, it was also shown that the
isoclasses of $\Pi$-stable $\mathbb{Z}Q$ configurations (two
configurations $\mathcal{C}$ and $\mathcal{C}'$ of $\mathbb{Z}Q$ are
called isomorphic if $\mathcal{C}$ is mapped onto $\mathcal{C}'$
under an automorphism of $\mathbb{Z}Q$) correspond bijectively to isoclasses of
RFS algebras of tree class Q with admissible group $\Pi$, except in
the case of $Q=D_{3m}$ with underlying field having characteristic
$2$. In such a case, each configuration corresponds to two
(isoclasses of) RFS algebras; both are symmetric algebras, one of
which is standard, while the other one is non-standard. Here, a
representation-finite $k$-algebra $A$ is called {\it standard} if
$k(\Gamma_A)$ is equivalent to ind$A$, where $\Gamma_A$ is the
AR-quiver of $A$ and ind$A$ is the full subcategory of mod$A$ whose
objects are specific representatives of the isoclasses of
indecomposable modules. This implies that any other standard RFS
algebras with AR-quiver isomorphic to $\Gamma_A$ is isomorphic to
$A$. Non-standard algebras are algebras which are not standard. The
non-standard algebras also have been studied by Waschb\"{u}sch in
\cite{Waschbusch1981}.  Note that when $A$ is
standard, then $k({}_s\Gamma_A)\simeq \underline{\mathrm{ind}}A$,
where $\underline{\mathrm{ind}}A$ is the full subcategory of
$\underline{\mathrm{mod}}A$ whose objects are objects in ind$A$;
while in case that $A$ is non-standard, $k{}_s\Gamma_A/J\simeq\underline{\mathrm{ind}}A$, where $k{}_s\Gamma_A$ is the path
category of ${}_s\Gamma_A$ and the ideal $J$ is defined by some
modified mesh relations (see \cite{Riedtmann4, Asashiba1999}).

We are now going to collect results from
\cite{BG,BLR,Riedtmann1,Riedtmann2,Riedtmann4} to show that a
configuration of RFS algebra $A$ gives a unique (weakly) sms, that
is, the two notions really coincide. From Definitions {\ref{wsms}}
and \ref{combinatorial-configuration}, the only difference between
them is that the homomorphism space is taken in the stable module
category and the mesh category {of its (universal) covering}
respectively. Hence to show the two notions are the same, it is
enough to show that the homomorphism spaces required in the two
definitions are isomorphic {upon restriction to mesh category of the
AR-quiver}.

\begin{Def}\label{well-behaved-functor} {\rm(\cite{Riedtmann1,Riedtmann2})}
Let $\pi: \Delta\rightarrow \Gamma$ be a covering where $\Gamma$ is
the AR-quiver (or stable AR-quiver) of $A$. The Auslander algebra
$E_A$ of $A$ is $End_A(M)$ where $M= \oplus_i M_i$ with each $M_i$ a
representative of an isoclass of indecomposable $A$-module. A
$k$-linear functor $F: k(\Delta)\rightarrow \mathrm{ind}A \mbox{ (or
}\underline{\mathrm{ind}}A)$ is said to be well-behaved if and only
if
\begin{enumerate}
\item For any $e\in \Delta_0$ with $\pi e = e_i$, we have $Fe = M_i$ where
$M_i$ is the indecomposable $A$-module corresponding to $e_i$;
\item For any $e\stackrel{\alpha}{\rightarrow}f$ in $\Delta_1$, $F\alpha$
is an irreducible map.
\end{enumerate}
\end{Def}

By \cite[Example 3.1b]{BG}, for any RFS algebra $A$ (whenever $A$ is
standard or non-standard), there is a well-behaved functor $F:k(\widetilde{\Gamma}_A)\rightarrow \mathrm{ind}A$ such that $F$ coincides with $\pi$ on objects,
where $\pi:\widetilde{\Gamma}_A\rightarrow \Gamma_A$ is the universal covering of the AR-quiver $\Gamma_A$. By \cite[Section 2.3]{Riedtmann1}, a
well-behaved functor is a covering functor and therefore there is a bijection
$$
\bigoplus_{Fh=Ff}Hom_{k(\widetilde{\Gamma}_A)}(e,h)\simeq Hom_A(Fe,Ff)
$$
for any $e,f,h\in (\widetilde{\Gamma}_A)_0$. Since an irreducible
morphism between non-projective indecomposable remains irreducible
under the restriction
$\mathrm{ind}A\rightarrow\underline{\mathrm{ind}}A$, the
well-behaved functor $F:k(\widetilde{\Gamma}_A)\rightarrow
\mathrm{ind}A$ restricts to a well-behaved functor
$\overline{F}:k({}_s\widetilde{\Gamma}_A)\rightarrow\underline{\mathrm{ind}}A$,
where ${}_s\widetilde{\Gamma}_A$ is the stable part of the
translation quiver $\widetilde{\Gamma}_A$. Note that the restriction
$\pi: {}_s\widetilde{\Gamma}_A\rightarrow{}_s\Gamma_A$ is also a
covering of the stable AR-quiver $_s\Gamma_A$. It follows that there
are bijections:

$$
\bigoplus_{Fh=Ff}Hom_{k({}_s\widetilde{\Gamma}_A)}(e,h)\simeq
\underline{Hom}_A(Fe,Ff);
$$

$$
\bigoplus_{\pi h=\pi f}Hom_{k({}_s\widetilde{\Gamma}_A)}(e,h)\simeq Hom_{k({}_s{\Gamma}_A)}(\pi e,\pi f).
$$

This implies:

\begin{Thm}\label{sms-configuration}
Let $A$ be an RFS algebra over an algebraically closed field.  Then
there is a bijection:
$$
\{\mbox{Configurations of }{}_s\Gamma_A \} \leftrightarrow
\{\mbox{sms's of } \stmod A\}$$

$$\quad  \quad  \quad \quad \quad \mathcal{C} \mapsto H(\mathcal{C})
$$
where the map $H:k({}_s\Gamma_A)\to \stmod A$ is defined on the
objects of the respective categories, and given by composing the
well-behaved functor
$\overline{F}:k({}_s\widetilde{\Gamma}_A)\rightarrow\underline{\mathrm{ind}}A$ with the natural embedding
$\underline{\mathrm{ind}}A$ into $\stmod A$.
\end{Thm}

\begin{Rem}\label{ARquiver-sms-question}
(1) {This theorem shows that all sms's of an RFS algebra $A$ can be
determined from the stable AR-quiver $_s\Gamma_A$, even in
non-standard case. }

(2) {This theorem also shows that $_s\Gamma_A$ determines sms$(B)$
for all indecomposable self-injective algebra $B$ such that
$_s\Gamma_B \simeq {}_s\Gamma_A$.  In fact, such phenomenon also
appears in the following tame case {(see \cite{KL})}: There is an
infinite series of $4$-dimensional weakly symmetric local algebras
$k\langle x,y\rangle /\langle xy-qyx\rangle$ for $q\in k^\times$
which have isomorphic stable AR-quivers, and are not stably
equivalent to each other. Their respective sms's are located in the same positions in
the stable AR-quivers of these algebras. }
\end{Rem}
In order to solve the simple-image problem of sms, we use a stable
equivalence classification of RFS algebras. This has been achieved
by Asashiba \cite{Asashiba2003}.  Before stating his result, we need
to define the type of $A$.  If $A$ is as above, by a theorem of
Riedtmann \cite{Riedtmann1}, $\Pi$ has the form $\langle \zeta\tau^{-r}\rangle$ where $\zeta$ is some automorphism of $Q$ and
$\tau$ is the translation. We also recall the Coxeter numbers of
$Q=A_n,D_n,E_6,E_7,E_8$ are $h_Q = n+1,2n-2,12,18,30$ respectively.
The frequency of $A$ is defined to be $f_A = r/(h_Q-1)$ and the
torsion order $t_A$ of $A$ is defined as the order of $\zeta$.  The
type of $A$ is defined as the triple $(Q,f_A,t_A)$.
{Note that the number of isoclasses of simple
$A$-modules is equal to $nf_A$.}

\begin{Thm}
\label{dclassRFS} {\rm(\cite{Asashiba1999,Asashiba2003})} Let $A$
and $B$ be RFS $k$-algebras for $k$ algebraically
closed.
\begin{enumerate}
\item If $A$ is standard and $B$ is non-standard, then $A$ and $B$ are not stably equivalent, and hence not derived equivalent.
\item If both $A$ and $B$ are standard, or both non-standard, the
following are equivalent:
\begin{enumerate}
\item $A,B$ are derived equivalent;
\item $A,B$ are stably equivalent of Morita type;
\item $A,B$ are stably equivalent;
\item $A,B$ have the same stable AR-quiver;
\item $A,B$ have the same type.
\end{enumerate}
\item The types of standard RFS algebras are the following:
\begin{enumerate}
\item $\{ (A_n,s/n,1) | n,s\in \mathbb{N}\}$,
\item $\{ (A_{2p+1},s,2) | p,s\in \mathbb{N}\}$,
\item $\{ (D_n,s,1) | n,s\in \mathbb{N}, n\geq 4\}$,
\item $\{ (D_{3m},s/3,1) | m,s\in \mathbb{N} , m\geq 2, 3\nmid s\}$,
\item $\{ (D_n,s,2) | n,s\in \mathbb{N}, n\geq 4\}$,
\item $\{ (D_4,s,3) | s\in \mathbb{N}\}$,
\item $\{ (E_n,s,1) | n=6,7,8; s\in \mathbb{N}\}$,
\item $\{ (E_6,s,2) | s\in \mathbb{N}\}$.
\end{enumerate}
Non-standard RFS algebras are of type $(D_{3m},1/3,1)$ for some
$m\geq 2$.
\end{enumerate}
\end{Thm}
\begin{Rem}
The RFS types which correspond to symmetric algebras are $\{(A_n,s/n,1) | s\in \mathbb{N},s\mid n\}$, $\{(D_{3m},1/3,1)|m\geq 2\}$,
$\{(D_n,1,1)|n\in \mathbb{N}, n \geq 4\}$ and $\{(E_n,1,1)|n=6,7,8\}$.
\end{Rem}

Combining these results with the fact that each configuration
corresponds to a set of simple modules of a (unique) stably
equivalent algebra and the fact that configuration and sms are the
same notion, the simple-image problem for RFS algebras has a
positive answer as stated above.  More precisely, let $\mathcal{S}$
be an {sms} of $A$. Then $\mathcal{S}$ corresponds to a configuration
$\mathcal{C}$ in the stable AR-quiver $_s\Gamma_A$, which
corresponds to $\mathcal{S}_B$ for some algebra $B$ with
$_s\Gamma_B\simeq {}_s\Gamma_A$.  This isomorphism between stable
AR-quivers then induces an equivalence $k(_s\Gamma_B)\to k(_s\Gamma_A)$ (or an equivalence $k_s\Gamma_B/J\to k_s\Gamma_A/J$
in case that $A$ and $B$ are non-standard).

Now this equivalence induces an equivalence
$\underline{\mathrm{ind}}B\to \underline{\mathrm{ind}}A$ and
consequently, an equivalence $\stmod B\to \stmod A$, with the
property that it sends $\mathcal{S}_B$ to $\mathcal{S}$.  This
finishes the proof of Theorem \ref{transsms}.
\bigskip

As a by-product of using configurations, we can pick out the
RFS algebras for which the transitivity problem raised in \cite{KL} has a
positive answer. That is, we can decide whether given two sms's of an algebra
there always is a stable self-equivalence sending the first sms to the second one.

\begin{Prop}\label{transitive-stable-self-equivalence}
If $A$ is an RFS algebra in the following list, then for any pair of
sms's $\mathcal{S}, \mathcal{S}'$ of $A$, there is a stable
self-equivalence $\phi:\stmod{A}\to\stmod{A}$ such that
$\phi(\mathcal{S})=\mathcal{S}'$. The list consists of $\{(A_2, s/2
, 1)| s\geq 1\}$, $\{(A_n, s/n, 1)| n\geq 1, \gcd(s,n)=1\}$,
$\{(A_3, s,2)| s\geq 1\}$, $\{(D_6, s/3,1)| s\geq 1, 3\nmid s\}$,
$\{(D_4, s,3)| s\geq 1\}$.
\end{Prop}

\begin{proof}
$A$ is an RFS algebra satisfying the condition stated if and only if the set of its sms's
 modulo the action of stable self-equivalences (i.e. the set of orbits of sms's under stable self-equivalences) is of size 1.
Every stable self-equivalence induces an automorphism of the stable AR-quiver $_s\Gamma_A=\mathbb{Z}Q/\Pi$ of $A$.
Conversely, any automorphism of $_s\Gamma_A$ induces a self-equivalence of $k(_s\Gamma_A)$ or of $k{}_s\Gamma_A/J$,
depending on $A$ being standard or not. Hence it induces stable self-equivalences of \underline{ind}$A$,
and consequently of $\stmod{A}$.  Therefore, identifying an sms with a configuration using Theorem \ref{sms-configuration},
the algebras $A$ we are looking for are those whose set Conf($_s\Gamma_A$)/Aut($_s\Gamma_A$) has just one element.
Here Conf($_s\Gamma_A$) is the set of configurations of $_s\Gamma_A$.  We now look at the number of Aut($_s\Gamma_A$)-orbits case by case.

For $E_n$ cases, one can count explicitly from the list of configurations in \cite{BLR} that the number of Aut($_s\Gamma_A$)-orbits are always greater than 1.

Now consider class $(A_n,s/n,1)$, $_s\Gamma_A=\Z A_n/\langle
\tau^s\rangle$.  Note that configurations of $\Z A_n$ are
$\tau^{n\Z}$-stable, so any configuration of $(A_n,s/n,1)$ are
$\tau^{d\Z}$-stable with $d=\gcd(s,n)$.  Let $s=ld$ and $n=md$. The
above implies configurations of $(A_n,l/m,1)$ are the same as
configurations of $(A_n,1/m,1)$.
But the number of the configurations of $(A_n,1/m,1)$ is equal to
the number of Brauer trees with $d$ edges and multiplicity $m$,
which is equal to 1 if and only if the pair $(d,m)=(2,1)$ or $d=1$.
Therefore, $(d,m)=(2,1)$ gives $\{(A_2,1,1)\}$, and $d=1$ yields the
family $\{(A_m,1/m,1)\}$.

Let $n=2p+1$. For the class $(A_n,s,2)$, $_s\Gamma_A=\Z A_n/\langle
\zeta\tau^{sn}\rangle$. A configuration of $(A_n,s,2)$ is
$\tau^{n\Z}$-stable as it is also a configuration of $\Z A_n$. So we
only need to consider the case $s=1$.
  Recall from
\cite[Lemma 2.5]{Riedtmann4} that there is a map which takes
configurations of $\Z A_n$ to configurations of $\Z A_{n+1}$, so the
numbers of orbits of $(A_n,1,2)$-configurations form an increasing
sequence. Therefore, we can just count the orbits explicitly.
$(A_3,1,2)$ has one orbit of configurations given by the
representative $\{(0,1),(1,2),(2,3)\}$, whereas $(A_5,1,2)$ has two
orbits.  This completes the $A_n$ cases.

Note that configuration of $\Z D_n$ is $\tau^{(2n-3)\Z}$-stable, so
similar to $A_n$ case we can reduce to the cases $(D_n,1,1),
(D_n,1,2), (D_4,1,3)$, and $(D_{3m},1/3,1)$. We make full use of the
main theorem in \cite{Riedtmann3} combining with our result in the
$A_n$ cases. Part (a) of the theorem implies that $(D_n,1,1)$ and
$(D_n,1,2)$ with $n\geq 5$ all have more than one orbits.  Part (c)
of the theorem implies that $(D_4,1,1)$ and $(D_4,1,2)$ has two
orbits, with representatives $\{(0,1), (1,1), (3,3), (3,4)\}$ and
$\{(0,2),(3,3),(3,4),(4,1)\}$. Since the latter is the only orbit
which is stable under the order 3 automorphism of $\Z D_4$, implying
$\{(D_4,s,3)| s\geq 1\}$ is on our required list. Finally, for
$(D_{3m},1/3,1)$ case, we use the description of this class of
algebras from \cite{Waschbusch1981}, which says that such class of
algebra can be constructed via Brauer tree with $m$ edges and
multiplicity 1 with a chosen extremal vertex.  Therefore, the only
$m$ with a single isomorphism class of stably equivalent algebra is
when $m=2$, hence giving us $\{(D_6,s/3,1)|s\geq 1, 3\nmid s\}$.
\end{proof}

\begin{Rem}
By the classification of RFS algebras due to Riedtmann and to
Bretscher, L\"aser and Riedtmann, the set
Conf($_s\Gamma_A$)/Aut($_s\Gamma_A$) is in bijection with the set
StAlg($A$) of Morita equivalence classes of algebras stably
equivalent to $A$ (Remark \ref{rmk-RiedtmannUniqueAlg}). Hence this
is also the list of RFS algebras for which StMAlg($A$) is 1, due to
Asashiba's theorem \ref{dclassRFS} (2).
\end{Rem}

\section{Sms's and Nakayama-stable smc's}
Our aim in this section is to prove that for an RFS algebra
$A$, every sms of $A$ lifts to a Nakayama-stable smc of
$D^b$(mod$A$), i.e. all sms of $A$ are liftable simple-image, which
proves the second assertion in Theorem B. We first state the results
and some consequences; the second part of this section then provides
the proof of the following result:

\begin{Thm} \label{lifting-theorem} Let $A$ be an RFS $k$-algebra over $k$ algebraically closed.
Then every sms $\mathcal{S}$ of $A$ is simple-image of Morita type under a liftable StM.
 Moreover, if $A$ and $B$ are two standard RFS
algebras, then every stable equivalence between $A$ and $B$ lifts to a standard
derived equivalence, and hence in particular, it is of Morita type.
\end{Thm}

As a consequence we get the second assertion in Theorem B:
\begin{Thm} \label{RFS-sms-smc-bijec} Let $A$ be an RFS
algebra over $k$ algebraically closed.  The map from smc$(A)$/DPic($A$) to sms$(A)$/StPic($A$) in
Theorem \ref{sms-smc} is a bijection. In particular, every sms
$\mathcal{S}$ of $A$ lifts to a Nakayama-stable smc of $D^b$(mod$A$).
\end{Thm}

\begin{Proof} By Theorem \ref{sms-smc}, it is enough to show that every sms $\mathcal{S}$ of
$\stmod A$ is a liftable simple-image:  There exists an algebra $B$
and a StM $\overline{\phi}:\stmod B\rightarrow \stmod A$ such that
$\overline{\phi}$ sends simple $B$-modules onto $\mathcal{S}$ and
that $\overline{\phi}$ lifts to a derived equivalence $\phi: D^b$(mod$B)\rightarrow D^b$(mod$A$).
{But this} follows from Theorem \ref{transsms} and Theorem \ref{lifting-theorem}.
\end{Proof}

Using Theorem \ref{RFS-sms-smc-bijec}, we can also strengthen
Theorem \ref{transsms} solving the stronger version of simple-image
problem, and hence completing Theorem B:

\begin{Cor} \label{number-orbits-sms} Let $A$ be an RFS algebra.  The map
StMAlg$(A)\to$sms$(A)/$StPic$(A)$ constructed in Theorem
\ref{orbits-sms} is a bijection. In particular, the number of
Morita equivalence classes of algebras which are StM to $A$ is the
same as the number of the orbits of sms's of $\stmod A$ under the
action of the stable Picard group of $A$.
\end{Cor}

Now that the strong simple-image problem has been answered, the following is an immediate consequence of Theorem \ref{orbits-sms}:
\begin{Cor} \label{stronger-transsms}
Every sms $\mathcal{S}$ of an RFS algebra $A$ over $k$ algebraically closed is simple-image of Morita type under some StM $\phi:\stmod B\to \stmod A$,
where the algebra {$B$ is unique up to isomorphism.}
\end{Cor}
\begin{Proof}
The uniqueness of $B$ follows from {Proposition \ref{unique-algebra}} once we have Theorem \ref{lifting-theorem}.
\end{Proof}

Combining Corollary \ref{stronger-transsms} with Theorem \ref{sms-smc} implies the following result which was not expected from the definition of sms's.

\begin{Cor} Let $A$ be an RFS algebra over $k$ algebraically closed. Then every sms $X_1,\cdots,X_r$ over $A$ is
Nakayama-stable, that is, the Nakayama functor $\nu_A$ permutes
$X_1,\cdots,X_r$.
\end{Cor}
\begin{Proof} An sms
$\mathcal{S}=\{X_1,\cdots,X_r\}$ over an RFS algebra $A$ can be lifted to a
Nakayama-stable smc of $D^b$(mod$A$).
\end{Proof}

In \cite[Section 6]{KL}, the following question has been posed: Is the cardinality of each sms over an artin algebra $A$
equal to the number of non-isomorphic non-projective
simple $A$-modules? A positive answer of this question implies the
Auslander-Reiten conjecture for any stable equivalence related to $A$. We answer
this question positively for RFS algebras.

\begin{Cor} Let $A$ be an RFS algebra over $k$ algebraically closed.
Then the cardinality of each sms over $A$ is equal to the number of
non-isomorphic simple $A$-modules.
\end{Cor}
\begin{Proof} By Theorem \ref{RFS-sms-smc-bijec}, every sms $\mathcal{S}$ of $\stmod A$ lifts to a
Nakayama-stable smc of $D^b$(mod$A$), and the cardinality of a
Nakayama-stable smc must be equal to the number of (isoclasses of)
simple modules by Rickard's or Al-Nofayee's result (cf. the proof of
Theorem \ref{sms-smc}).

Alternatively, using Theorem \ref{sms-configuration}, all sms's of
$A$ correspond to configurations, which are all finite and have the
same cardinality, equal to the number of isoclasses of simple
$A$-modules.
\end{Proof}

Validity of the Auslander-Reiten conjecture in this case first has been shown in \cite{Riedtmann1}. By results
of Martinez-Villa \cite{Martinez-Villa} the conjecture is valid for all representation finite algebras.

\subsection{Proof of Theorem \ref{lifting-theorem}}
It remains to prove Theorem \ref{lifting-theorem}. The proof will
occupy the rest of this section. It will be subdivided in a first part dealing with standard RFS algebras,
and a second part dealing with the non-standard case.
\medskip

{\it The standard case.}
\smallskip

For standard RFS algebras, Asashiba \cite{Asashiba2003}
already solved this problem in most, but not all cases. We first
recall Asashiba's description of stable Picard groups for standard
RFS algebras.

\begin{Thm} \label{stable-Picard-group-RFS} {\rm(\cite{Asashiba2003})}
Let $A$ be a standard RFS algebra.
If $A$ is not of type $(D_{3m},s/3,1)$ with $m\geq 2,3\nmid s$, then
$$StPic(A)=Pic'(A)\langle[\Omega_A]\rangle.$$
If $A$ is of type $(D_{3m},s/3,1)$ with $m\geq 2,3\nmid s$, then
$$StPic(A)=(Pic'(A)\langle[\Omega_A]\rangle)\cup(Pic'(A)\langle[\Omega_A]\rangle)[H],$$
where $H$ is a stable self-equivalence of $A$ induced from an
automorphism of the quiver $D_{3m}$ by swapping the two high
vertices; it satisfies $[H]^2\in Pic'(A)$.
\end{Thm}

\begin{Rem}
See \cite{Riedtmann4} and \cite{BLR} for an explanation of the concept of
high vertices.
Note that the stable Picard group here, by definition, contains
\textit{all} stable self-equivalences, rather than as usual
only the stable self-equivalences {of Morita type}. From
the description below, in the standard case all stable self-equivalences are of Morita type. So the choice of another, possibly different,
stable Picard group does not matter here.
\end{Rem}

By the description in Theorem
\ref{stable-Picard-group-RFS}, if a standard RFS algebra $A$ is not of type $(D_{3m},s/3,1)$ with $m\geq 2,3\nmid s$, then every stable self-equivalence
over $A$ is of Morita type. If $A$ is of type $(D_{3m},s/3,1)$ with $m\geq 2,3\nmid s$, then
every stable self-equivalence over $A$ can be
determined by the image of objects of $\underline{\mathrm{ind}}A$, up to some Morita equivalence which
induces the identity map on the objects of $\underline{\mathrm{ind}}A$
(c.f. description of stable AR-quiver of type $(D_{3m},1/3,1)$ after Example \ref{Example}).
Examples of such Morita equivalences include $H^2$, and the following:

\begin{Ex}\label{determine-by-object-shifted}
Let $A$ be an RFS algebra of type $(D_{3m},1/3,1)$ with $m\geq 2$
and char $k \neq 2$. Then $\Omega^{2m-1}$ fixes all indecomposable
objects and $2m-1$ is the smallest positive exponent of the Heller
shift for which this works. But $\Omega^{2m-1}$ is not naturally
isomorphic to the identity functor on the stable module category.
For getting that, the smallest positive exponent of the Heller shift
is $2(2m-1)$. When char $k=2$, we do have $\Omega^{2m-1}\simeq \mathrm{id}_{\stmod A}$.
\end{Ex}

\begin{Rem}\label{determine-by-object-remark}
(1) Similar argument works for type $(D_{3m},s/3,1)$ with $m\geq 2, 3\nmid s$. Note that for this family of RFS
algebras, a general stable self-equivalence has the form $\phi = F$
or $\phi = FH$ where $F\in Pic'(A)\langle[\Omega_A]\rangle$. If
$\phi$ can be lifted to a standard derived equivalence, then $H$ can
be lifted as well, since $\Omega$ is lifted to the inverse
suspension functor $[-1]$ on the derived category.

(2) There may not exist any $d>0$ with $\Omega^d\simeq \mathrm{id}_{\stmod A}$. Precise information when this does (or does
not) occur, and for which $d$, can be found in \cite[Theorem 6.1]{Dugas}.
\end{Rem}

One important application of Asashiba's Theorem
\ref{stable-Picard-group-RFS} is that we can now lift any stable
equivalence between two RFS algebras as long as they are not of type
$(D_{3m},s/3,1)$ with $m\geq2$ and $3\nmid s$ \cite[Main
Theorem]{Asashiba2003}.  The reason why this result did not cover
the type $(D_{3m},s/3,1)$ with $m\geq 2,3\nmid s$ is that the
liftability of stable self-equivalence $H$ was not known until a
recent proof in \cite{Dugas2011}. In the following, we will use
another result of Dugas in \cite{Dugas2012} concerning mutation of
sms's to prove that $H$ lifts indeed to a standard derived
equivalence.  This extends the main theorem of \cite{Asashiba2003}
to all standard RFS algebras, and consequently allowing us to prove
Theorem \ref{lifting-theorem} in the standard case.  We start by
recalling the mutation of sms from \cite{Dugas2012}. Our definition
here is a variation of Dugas's original one by shifting the objects
by $\Omega^{\pm 1}$, so that the mutations ``align" with the
mutation for smc defined in \cite{KY} (see \cite{Dugas2012} Remark
under Definition 4.1, \cite{KY} and Section 5 for more details). We
restrict to the stable category of a self-injective algebra,
although the original definition works for more general triangulated
categories.  For the definitions of left/right approximations see
for example \cite{Aihara2012, AI, KY, Dugas}.

\begin{Def} \label{mutation-sms} {\rm(\cite[Definition 4.1 and Remark]{Dugas2012})} Let $A$ be a finite-dimensional self-injective
algebra and $\mathcal{S} = \{X_1,\ldots, X_r\}$ an sms of $A$.
Suppose that $\mathcal{X}\subseteq \mathcal{S}$ is a Nakayama-stable
subset: $\nu_A(\mathcal{X}) =\mathcal{X}$. Denote by $\mathcal{F(X)}$ the
smallest extension-closed subcategory of $\stmod A$ containing
$\mathcal{X}$.
 The left mutation of the sms $\mathcal{S}$ with respect to $\mathcal{X}$ is
the set
  $\mu_{\mathcal{X}}^+(\mathcal{S}) = \{Y_1,\ldots,Y_r\}$ such that
\begin{enumerate}
\item $Y_j=\Omega^{-1}(X_j)$, if $X_j\in \mathcal{X}$
\item Otherwise, $Y_j$ is defined by the following distinguished triangle
$$\Omega(X_j)\rightarrow X \rightarrow Y_j,$$ where the first map is a minimal left $\mathcal{F(X)}$-approximation
of $\Omega(X_j)$.
\end{enumerate}
The right mutation $\mu_{\mathcal{X}}^-(\mathcal{S})$ of
$\mathcal{S}$ is defined similarly.
\end{Def}

It has been shown in \cite{Dugas2012} that the above defined sets
$\mu_{\mathcal{X}}^+(\mathcal{S})$ and
$\mu_{\mathcal{X}}^-(\mathcal{S})$ are again sms's. This definition
works for all self-injective algebras as long as
$\nu(\mathcal{X})=\mathcal{X}$, which is automatically true for
weakly symmetric algebras.  Mutation of sms is designed to keep
track of the images of simple modules (which form an sms) under (liftable) StM. It is interesting to ask if all sms's can be obtained just by mutations;
{this will be considered
in Section 5.}

\begin{Ex}\label{Example}
Let $A$ be a symmetric Nakayama algebra with 4 simples and Loewy
length 5. The canonical sms is the set of simple $A$-modules
$\{1,2,3,4\}$. The left mutation of $\mathcal{S}$ at
$\mathcal{X}=\{2,3\}$ is
\[
\mu_{2,3}^+(1,2,3,4) = \{\begin{array}{ccc} 1 \\
 2 \\ 3
\end{array}, \begin{array}{cccc} 2 \\
 3\\ 4\\ 1\end{array}, \begin{array}{cccc} 3 \\
 4\\ 1\\ 2\end{array}, \begin{array}{c} 4
\end{array}\}.
\]
\end{Ex}

Now let $A$ be a standard RFS algebra of type $(D_{3m},s/3,1)$ with
$m\geq 2,3\nmid s$. Without loss of generality, we
may assume that $A$ is the algebra representing this class, which has been given by Asashiba
\cite[Appendix 2]{Asashiba2003}, whose quiver is given in Figures
$Q(D_{3m}, s/3)$ below.
When $s = 1$, the stable AR-quiver
$_s\Gamma_A=\mathbb{Z}D_{3m}/\langle\tau^{(2m-1)}\rangle$ is given
by connecting $(2m-1)$ copies of $D_{3m}$. The position of the
indecomposable $A$-modules on $_s\Gamma_A$ can be found in
Waschb\"{u}sch \cite{Waschbusch1981}. The $m-1$ simple
modules lie on the mouth (boundary) of the stable tube; and the
remaining one lies in a high vertex (using terminology of Riedtmann
\cite{Riedtmann4} and BLR \cite{BLR}). When $s > 1$, the stable
AR-quiver $_s\Gamma_A=\mathbb{Z}D_{3m}/\langle\tau^{(2m-1)s}\rangle$
is given by connecting $s$ copies of stable AR-quiver of that in
case $s = 1$. Explicit calculations demonstrate the following observation on the simple
$A$-modules:

\begin{enumerate}
\item The vertices in the inner circle (loop path $\beta_s\cdots\beta_1$)
correspond to simple modules in the high vertex of the stable
AR-quiver, with $\tau^{(2m-1)}$ of such a simple being another such
simple. We label these vertices by $v_1,\cdots, v_s$, which can be
thought of as ramification of the vertex $v_1$ in the $s = 1$
case, see Figure $(D_{3m},1/3)$.
\item Let $i\in \{1,\cdots,s\}$, and consider vertices on the path
$\alpha^{(i)}_{m-1}\cdots\alpha^{(i)}_{2}$. There are $m-1$ such
vertices for each $i$, and we label these by $i_1,\cdots,i_{m-1}$.
The corresponding indecomposable projective modules are uniserial,
and the corresponding $m-1$ simple modules lie on the mouth of
$i$-th copy of stable AR-quiver (in the same way as in the $s = 1$
case).
\item The Nakayama functor permutes the simple $A$-modules as follows:
$$v_i \mapsto v_{i+3}$$ $$i_j\mapsto (i+3)_j$$
for all $i\in \{1,\cdots,s\}$ (where we think of $1-1 = s$) and all
$j\in \{1,\cdots,m-1\}$.
\end{enumerate}
\bigskip
$$
\vcenter{ \xymatrix@!R=0.2pc{ & 1_{m-1} \ar[ld]_{\al_m}&  \ar[l]_{\al_{m-1}} \ar@{{}*{\cdot}{}}[r] & \ar@/^1pc/@{{}*{\cdot}{}}[dd] \\
v_1 \ar[rd]_{\al_1} \ar@(ul,dl)[]_{\be}&&&&\\
& 1_1 \ar[r]_{\al_2}& ~ \ar@{{}*{\cdot}{}}[r] & }}
$$
\medskip
$$\mbox{Figure }Q(D_{3m}, 1/3)$$

\bigskip
$$ \vcenter{
\xymatrix@!C=0.1pc@!R=0.1pc@M0mm{
&&&&&\ar[lldd]_(0.1){\al_{m-1}^{(s)}}\ar@{{}*{\cdot}{}}[rr]&& \\
&&&\ar@{{}*{\cdot}{}}[lldd]&&
      \circ \ar[ll]_(0.8){\al_2^{(1)}}&&\circ \ar[ld]^{\al_m^{(s-1)}} &&
     \ar[ll]_(0.3){\al_{m-1}^{(s-1)}}\ar@{{}*{\cdot}{}}[rd] \\
&&& \circ \ar[d]^{\al_m^{(s)}}&&& \circ \ar[llld]^{\be_1}\ar[lu]_{\al_1^{(1)}} &&&
    \circ \ar[lluu]_(0.8){\al_2^{(s)}}&\\
&\ar[dd]_(0.3){\al_{m-1}^{(1)}}&\circ \ar[lldd]_(0.85){\al_2^{(2)}} &
     \circ \ar[l]_{\al_1^{(2)}}\ar[lddd]^{\be_2}&&&&&&
     \circ \ar[lllu]^{\be_s}\ar[u]^{\al_1^{(s)}}& \circ \ar[l]_{\al_m^{(s-2)}}&\\
&&&&&&&&&&\ar[lu]^{\be_{s-1}}\ar@<0.3ex>@{{}*{\cdot}{}}[d] &
      \ar[lu]_{\al_{m-1}^{(s-2)}}\ar@{{}*{\cdot}{}}[rd]\\
\ar@{{}*{\cdot}{}}[dd]& \circ \ar[rd]^{\al_m^{(1)}}&&&&&&&&&&\ar@{{}*{\cdot}{}}[dd]
      &\ar@{{}*{\cdot}{}}[lu]\\
&&\circ \ar[ld]^{\al_1^{(3)}}\ar[rddd]^{\be_3}&&&&&&&&&&&\\
\ar[rrdd]_(0.15){\al_{m-1}^{(2)}}&\circ \ar[dd]_(0.7){\al_2^{(3)}}&&&&&&&&&&&\\
&&&&&&&&&&\ar@<-0.3ex>@{{}*{\cdot}{}}[u]&\ar@{{}*{\cdot}{}}[ru]&\\
&\ar@{{}*{\cdot}{}}[rrdd]&\circ \ar[r]^{\al_m^{(2)}} &
     \circ \ar[d]_{\al_1^{(4)}}\ar[rrrd]^{\be_4}&&&&&&
     \circ \ar[ru]^{\be_6}\ar[r]_{\al_1^{(6)}}&\circ \ar[ru]_{\al_2^{(6)}} & \\
&&&\circ \ar[rrdd]_(0.8){\al_2^{(4)}}&&&\circ \ar[rrru]^{\be_5}\ar[rd]^{\al_1^{(5)}}
     &&&\circ \ar[u]^{\al_m^{(4)}}&\\
&&&\ar[rr]_(0.3){\al_{m-1}^{(3)}}&&\circ \ar[ru]^{\al_m^{(3)}}&&
     \circ \ar[rr]_(0.7){\al_2^{(5)}}&&\ar@{{}*{\cdot}{}}[ru]\\
&&&&&\ar@{{}*{\cdot}{}}[rr]&&\ar[rruu]_(0.2){\al_{m-1}^{(4)}}
}}
$$
\medskip
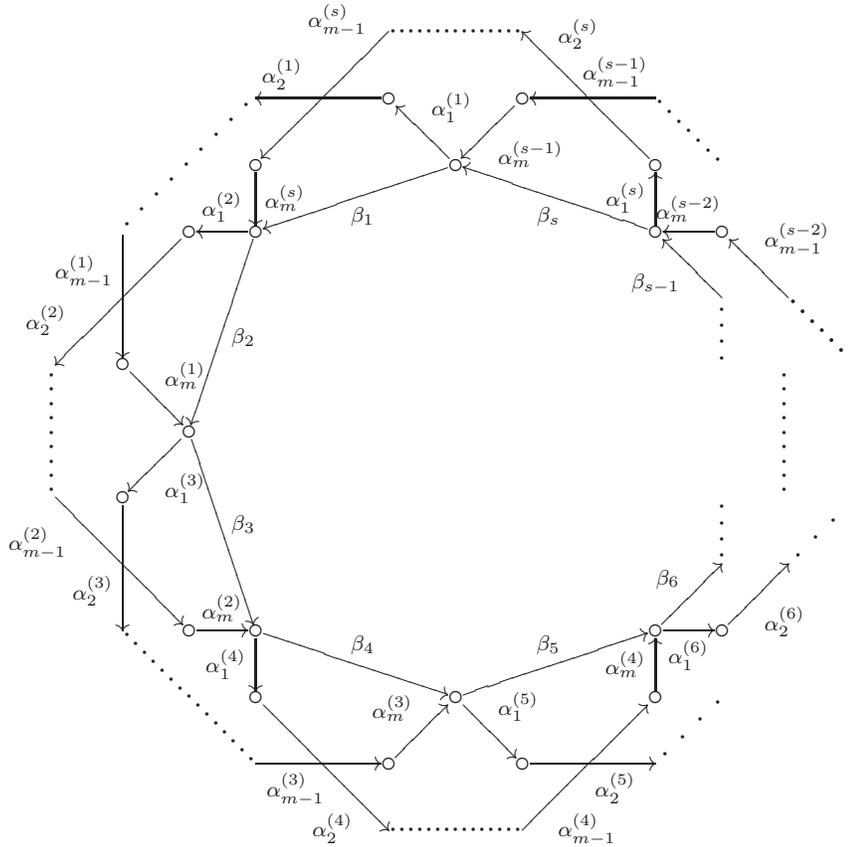
$$\mbox{Figure }(Q(D_{3m}, s/3), s\geq 2)$$
\medskip

We now mutate the sms $\mathcal{S}$ of simple $A$-modules at the
Nakayama-stable subset $\mathcal{X}=\{1_1,\cdots,s_{1}\}$. The above
observation implies that the left mutation
$\mu_{\mathcal{X}}^+(\mathcal{S}_A)$ has the same effect as the
composition $\tau\circ H$ of the stable self-equivalence $H$ (cf.
Theorem \ref{stable-Picard-group-RFS}) and $\tau$ on
$\mathcal{S}_A$, where $\tau=\nu_A\circ{\Omega_A}^2$ is the
AR-translation. According to
 Dugas \cite[Section 5]{Dugas2012}, $\mu_{\mathcal{X}}^+(\mathcal{S}_A)$ can be
realised by a derived equivalence $\phi:D^b$(mod$B)\rightarrow
D^b$(mod$A$) for some standard RFS algebra $B$ with the same type.
Using the same argument as in the proof of Theorem \ref{sms-smc}, we
can assume that $\phi$ is a standard derived equivalence. Then
$\phi$ induces a StM $\overline{\phi}:\stmod B\rightarrow \stmod A$
such that $\overline\phi$ sends simple $B$-modules to
$\mu_{\mathcal{X}}^+(\mathcal{S}_A)$, which coincides with the image
of $\tau\circ H$ on $\mathcal{S}_A$. Next we show that the algebra
$B$ is isomorphic to $A$, and therefore $\overline{\phi}$ can be
identified as a stable self-equivalence over $A$. In fact, by
Riedtmann's classification on RFS algebras, there is a bijection
between Conf$(_s\Gamma_A)/Aut(_s\Gamma_A)$ and StAlg$(A)$ as in
proof and remark of Proposition
\ref{transitive-stable-self-equivalence}.  Since
$\overline{\phi}(\mathcal{S}_B)=\tau\circ H(\mathcal{S}_A)$ is in
the same orbit as $\mathcal{S}_A$ in
Conf$(_s\Gamma_A)/Aut(_s\Gamma_A)$, we have $[B]=[A]$ in StAlg$(A)$,
therefore $B$ is isomorphic to $A$. {Since $\overline\phi$ can be
lifted to a standard derived equivalence, it follows from our
previous discussion in Example \ref{determine-by-object-shifted} and
Remark \ref{determine-by-object-remark} that $H$ can also be lifted,
and hence it is a liftable StM.}

The above result has been proved by Dugas \cite[Section 5]{Dugas2011} at least for the case $s=1$. Our proof here is carried
out in the same spirit as his, but with the point of view focussing
on configurations which clarifies the ``covering technique"
mentioned in Dugas' article when he generalises the result to $s>1$
case.  In particular, we avoid calculating explicitly the algebra
$B$ which was the approach used in \textit{loc. cit.}  We have
finished the proof of Theorem \ref{lifting-theorem} in the standard
case.
\medskip

{\it The non-standard case.}
\smallskip
Now we prove Theorem \ref{lifting-theorem} in the
non-standard case. Let $A$ be a non-standard RFS algebra of type
$(D_{3m},1/3,1)$ and let $A_s$ be its standard counterpart. First we
recall some facts:
\begin{enumerate}
\item (standard-non-standard correspondence): There is a bijection ind($A$) $\leftrightarrow$ ind($A_s$)
between the set of indecomposable objects and irreducible morphisms, which is compatible with the position on the stable AR-quiver
$\Gamma = \Z D_{3m}/\langle \tau^{2m-1}\rangle$. More precisely, by Waschb\"{u}sch \cite{Waschbusch1981},
the AR-quiver of $A$ is obtained from that of $A_s$ by replacing every part of the Loewy diagram

$$
\xymatrix@C=0.4pc@R=0.1pc@M=1mm{
1_m & & & 1_m \ar@{-}[ld] \ar@{-}[dd] \\
v_1 & \text{to} & v_1 \ar@{-}[rd] & \\
v_1 & & & v_1 \\
1_1 & & & 1_1}
$$

\item There is one-to-one correspondence between the following three
sets:
$$
sms(A) \leftrightarrow Conf(\Gamma) \leftrightarrow sms(A_s)
$$
where the first is the set of sms's of $A$, the second is the set of
configurations of $\Gamma$, and the third is sms's of $A_s$.
\item If $B$ is another non-standard RFS algebra of type
$(D_{3m},1/3,1)$, then there is a liftable StM $\phi:\stmod{A}\to\stmod{B}$
(see Theorem \ref{dclassRFS}).
\end{enumerate}

Therefore, by $(3)$, we can assume $A$ is the representative of the class of algebras
of type $(D_{3m},1/3,1)$, whose quiver is also given in Figure $Q(D_{3m}, 1/3)$.

\begin{Lem}\label{lem1}
Every stable self-equivalence $\phi_s\in StPic(A_s)$ has a non-standard counterpart $\phi \in StPic(A)$ such that,
if $\phi_s$ maps the set $\mathcal{S}_{A_s}$ of simple $A_s$-modules to $\mathcal{S}_s$, then $\phi(\mathcal{S}_A)=\mathcal{S}$ where $\mathcal{S}$
corresponds to $\mathcal{S}_s$ in the above correspondence.  Moreover, $\phi$ is a liftable StM.
\end{Lem}
\begin{proof}
By Asashiba's description, $StPic(A_s) = Pic'(A_s)\langle
\Omega\rangle[H]$.  If $\phi_s\in Pic'(A_s)$, then it must permute
the $m-1$ simple modules on the mouth of the stable tube and fixes
the remaining one in a high vertex. It follows from the description
of the stable AR-quiver of $A_s$ that $\phi_s$ fixes
$\mathcal{S}_{A_s}$ and induces the identity map $Conf(\Gamma)\to
Conf(\Gamma)$. Therefore we can simply pick the (liftable StM)
identity functor for $\phi$.  If $\phi_s = \Omega_{A_s}^n$ for some
$n\in\Z$, then by standard-non-standard correspondence, picking
$\phi$ to be the Heller shift $\Omega_A^n$ of $A$ will do the trick.
This is obviously a liftable StM.  For $H$, we do the same sms
mutation $\mu_{1_1}^+(\mathcal{S}_A)$ as in the standard case, which
gives a liftable StM $\phi$.
\end{proof}

\begin{Thm}\label{liftable-nonstandard}
Every sms $\mathcal{S}$ of $A$ is simple-image of Morita type under a liftable StM.
\end{Thm}
\begin{proof}
We know by Theorem \ref{transsms} that $\mathcal{S}$ is
simple-image, so there is some stable equivalence $\psi:\stmod B \to
\stmod A$ with $\psi(\mathcal{S}_B)=\mathcal{S}$.  By the above fact
$(3)$, there is a liftable StM $\phi_1:\stmod{B}\to \stmod{A}$. Let
$\mathcal{S}'=\phi_1(\mathcal{S}_B)$.
 If $\mathcal{S}'=\mathcal{S}$, then we are done.  Otherwise, their corresponding sms's $\mathcal{S}_s$ and $\mathcal{S}_s'$ of $A_s$ are also not equal.
 But they belong to the same $StPic(A_s)$-orbit, since $\phi_1$ induces an automorphism on the stable AR-quiver of $A$ or $A_s$,
 so there is some stable equivalence $\phi_s:\stmod{A_s}\to\stmod{A_s}$ sending $\mathcal{S}_s'$ to $\mathcal{S}_s$.
 This gives a liftable StM $\phi_2:\stmod{A}\to\stmod{A}$ by Lemma \ref{lem1}, and it maps $\mathcal{S}'$ to $\mathcal{S}$.
 Now we have a liftable StM $\phi=\phi_2\phi_1:\stmod{B}\to\stmod{A}$ with $\phi(\mathcal{S}_B)=\mathcal{S}$.
\end{proof}

This finishes the proof of Theorem \ref{lifting-theorem}.
\bigskip

\section{Sms's and mutations}
In this section, we discuss connections with mutations and with
tilting quivers and how to use these concepts for sms. A main result
is Theorem \ref{RFS-tilting-connected}, which states that the
homotopy category $\mathcal{T}=K^b(\mathrm{proj}A)$ is strongly
tilting-connected when $A$ is an RFS algebra. This result is
formally independent of sms, but it fits well with the point of view
taken in this paper.

The first connection we consider here comes from the aforementioned
result of Dugas \cite{Dugas2012}, which opens up a new and efficient way
to study (and compute) simple-image sms's of Morita type and their
liftability, as demonstrated in the previous section.

We have seen how mutation of sms and Nakayama-stable smc are
connected. We remind the reader of the main result of \cite{KY},
which in particular
gives a bijection between smc and silting objects as well as
compatibility of the respective mutations. Since we have already
established a connection between sms and smc, we can now exploit the
connection with silting / tilting objects.

First we briefly recall some information on silting theory developed
by Aihara and Iyama \cite{AI}. Throughout this section, $A$ is an
indecomposable non-simple self-injective algebra over an
algebraically closed field. We use $\mathcal{T}$ to denote the
(triangulated) homotopy category $K^b(\mathrm{proj}A)$ of bounded
complexes of projective $A$-modules; the suspension functor in this
category is denoted {by $[1]$, and by $[n]$ we mean $[1]^n$.}

\begin{Def}\label{silting-thy} (\cite{AI})
\begin{enumerate}
\item Let $T\in \mathcal{T}$. Then $T$ is a silting (resp. tilting) object if:
\begin{enumerate}
\item $Hom_{\mathcal{T}} (T,T[i])=0$ for any $i>0$ (resp. $i\neq 0$)
\item The smallest thick subcategory of $\mathcal{T}$ containing
$T$ is $\mathcal{T}$ itself.
\end{enumerate}
\item Let $T=X_1\oplus\cdots \oplus X_r$ be a silting object
({where each $X_i$ is indecomposable}) and
$\mathcal{X}\subset\{1,\ldots,r\}$. A left silting mutation of $T$
with respect to $\mathcal{X}$, denoted by $\mu_{\mathcal{X}}^+(T) = Y_1\oplus\cdots \oplus Y_r$ satisfies by definition that the
indecomposable summands $Y_i$ are given as follows:
\begin{enumerate}
\item $Y_i = X_i$ for $i\notin \mathcal{X}$
\item For $i\in \mathcal{X}$:
\[
Y_j := \mathrm{cone}(\text{minimal left }\mathrm{add}({\bigoplus_{i\notin
\mathcal{X}}X_i})\text{-approximation of }X_j)
\]
\end{enumerate}
A right silting mutation $\mu_{\mathcal{X}}^-$ is defined similarly
using right approximation.  A silting mutation is said to be
irreducible if {$\mathcal{X}=\{i\}$ for some $i$.}
\end{enumerate}
\end{Def}

Note that tilting objects in $\mathcal{T}$ (i.e. one-sided tilting
complexes) are exactly the silting objects that are stable under
Nakayama functor (see, for example, the discussion after Theorem 3.5 of \cite{KY}). As we have hinted throughout the whole article,
Nakayama-stability plays a vital role in the study of sms's, at
least for sms's which are liftable and simple-image of Morita type.
For convenience, we denote the Nakayama functor $\nu =\nu_A$ when
the algebra $A$ under consideration is clear, and we assume every
tilting object is basic, i.e. its indecomposable summands are
pairwise non-isomorphic.

\begin{Lem}\label{nakayama-stable-mutation}
{Let $A, \mathcal{T}$ be as above and $\mathcal{C}$ a full
subcategory of $\mathcal{T}$ with $\nu\mathcal{C} = \mathcal{C}$.}
If $Y\in \mathcal{T}$ and $f:X\to Y$ is a (minimal) left
$\mathcal{C}$-approximation of $Y$, then $\nu_A(f):\nu X\to \nu Y$
is a (minimal) left $\mathcal{C}$-approximation of $\nu Y$. In
particular, if $\nu Y=Y$, then $\nu X=X$.
\end{Lem}
\begin{proof}
Since $A$ is self-injective, so $\nu\mathcal{T} = \mathcal{T}$, and
$Hom_{\mathcal{T}}(X,Y) \simeq Hom_{\mathcal{T}}(\nu X,\nu Y)$. As
$\nu X\in \mathcal{C}$, to see $\nu f$ is a
$\mathcal{C}$-approximation, we need to show that
$Hom_{\mathcal{T}}(\nu f,X')$ is surjective for all
$X'\in\mathcal{C}$.  Since $\nu\mathcal{C}=\mathcal{C}$, every
object in $X'\in\mathcal{C}$ is of the form $\nu Z$ for some
$Z\in\mathcal{C}$.  Also $Hom_{\mathcal{T}}(\nu X,\nu Z) \simeq Hom_{\mathcal{T}}(X,Z)$, so every map $\nu X\to \nu Z$ can be
written as $\nu h$ for some $h: X\to Z$.  Since $f$ is an
approximation, $h = fg$ for some $g\in Hom_{\mathcal{T}}(Y,Z)\simeq Hom_{\mathcal{T}}(\nu Y,\nu Z)$.  As
$\nu$ is an auto-equivalence of $\mathcal{C}$, $\nu h= \nu (fg)=(\nu f)(\nu g)$. Hence $\nu f:\nu X\to \nu Y$ is a
$\mathcal{C}$-approximation.  For minimality we proceed similarly.
i.e. for $g: \nu X\to \nu X$, $g=\nu h$ for some $h:X\to X$, the
condition $g(\nu f)=\nu f$ can now be rewritten as $\nu (hf)=(\nu h)(\nu f) = \nu f$ which implies $hf=f$.  By minimality of $f$, $h$
is an isomorphism, hence so is $\nu h$.
\end{proof}

By this lemma, a mutation of a tilting object (i.e. a
Nakayama-stable silting object) is a tilting object if and only if
we mutate at a Nakayama-stable summand. Therefore, a tilting
mutation is the same as a silting mutation at a Nakayama-stable
summand.  {An irreducible silting mutation
mutates with respect to an indecomposable summand. By thinking of
this as mutating with respect to a ``minimal" Nakayama-stable
summand,} we can make sense of ``irreducibility" for tilting
mutation for general self-injective algebras (rather than just
weakly symmetric algebras).

\begin{Def} ({Compare to \cite{Aihara2012}})
(1) Let $T=T_1\oplus\cdots\oplus T_r$ be a
basic tilting object in $\mathcal{T}=K^b(\mathrm{proj}A)$. If $X$ is
a Nakayama-stable summand of $T$ such that for any Nakayama-stable
summand $Y$ of $X$, we have $Y=X$, then we call $X$ a minimal
Nakayama-stable summand.  A (left) tilting mutation $\mu_{X}^+(T)$
is said to be irreducible if $X$ is minimal. Similarly for right
tilting mutation.

(2) Let $T,U$ be basic tilting objects in $\mathcal{T}$. We say that
$U$ is connected (respectively, left-connected) to $T$ if $U$ can be
obtained from $T$ by iterative irreducible (respectively, left)
tilting mutations.

(3) $\mathcal{T}$ is tilting-connected if all its basic tilting
objects are connected to each other. We say that $\mathcal{T}$ is
strongly tilting-connected if for any basic tilting objects $T,U$
with $Hom_{\mathcal{T}}(T,U[i])=0$ for all $i>0$, $U$ is
left-connected to $T$.
\end{Def}
\begin{Rem}
(1) {Note that the irreducible tilting mutation just defined is
different from an irreducible silting mutation when $A$ is
self-injective non-weakly symmetric, even though it is itself a
silting mutation as well. We will emphasise irreducible
\textit{tilting} mutation throughout to distinguish between our
notion and irreducible silting mutation.}

(2) {We can define the analogous notion of (left or right)
irreducible sms mutation similar to irreducible tilting mutation
above. More precisely, for an sms $\mathcal{S} = \{X_1,\ldots,X_r\}$ as in Definition \ref{mutation-sms},
its irreducible mutation means that we mutate at a Nakayama-stable subset
$\mathcal{X} = \{X_{i_1},\ldots, X_{i_m}\}$ which is minimal in the obvious sense. }

(3) {Strongly tilting-connected implies tilting-connected. This
follows from the fact that left and right mutations are inverse
operations to each other, i.e. $\mu_Y^-\mu_X^+(T) = T =\mu_Z^+\mu_X^-(T)$ where $T=X\oplus M$, $\mu_X ^+(T)=Y\oplus M$, and $\mu_X^-(T)=Z\oplus M$.}

\end{Rem}

We can now reformulate a question asked in \cite{AI} and
\cite[Question 3.2]{Aihara2012}:  Is
$\mathcal{T}=K^b(\mathrm{proj}A)$ tilting-connected for
self-injective algebra $A$?  By reproving the Nakayama-stable
analogue of the results in \cite{AI} and \cite{Aihara2012}, we can
answer this question positively for RFS algebras $A$. These proofs
are not directly related to the simple-minded theories and are
really about modifying the proofs of Aihara and of Aihara and Iyama in an
appropriate way.

\begin{Thm}\label{RFS-tilting-connected}
Let $A$ be an RFS algebra. Then the homotopy category
$\mathcal{T}=K^b(\mathrm{proj}A)$ is strongly tilting-connected.
\end{Thm}

The proof will occupy a separate subsection below.

Recall the silting quiver as defined in \cite{AI} and
\cite{Aihara2012}. Again we can define a ``Nakayama-stable version"
and the sms's version of this combinatorial gadget.

\begin{Def} ({Compare to \cite{Aihara2012,AI}})
Let $A$ be a self-injective algebra.

(1)  Let tilt($A$) be the class of all tilting objects in
$\mathcal{T}=K^b(\mathrm{proj}A)$ up to shift and homotopy
equivalence.  The tilting quiver of $\mathcal{T}$ is a quiver
$Q_{\mathrm{tilt}}(A)$ such that the set of vertices is the class of
basic tilting objects of $\mathcal{T}$; and for $T,U$ tilting
objects, $T \to U$ is an arrow in the quiver if $U$ is an
irreducible left tilting mutation of $T$.

(2) Let sms($A$) denote the class of all sms's of $A$.  The mutation
quiver of sms($A$) is a quiver $Q_{\mathrm{sms}}(A)$ such that the
set of vertices is sms($A$); and for two sms's $\mathcal{S},\mathcal{S}'$, $\mathcal{S}\to \mathcal{S}'$ is an arrow in the
quiver if $\mathcal{S}'$ is {a mutation of $\mathcal{S}$}.
\end{Def}
\begin{Rem}
(1) Long before the work of \cite{AI}, the term tilting quiver has
been used for a graph whose vertices are tilting modules over a
finite dimensional algebra. The tilting quiver here is a
specialisation of the silting quiver of \cite{AI}, whose vertices
are objects in a triangulated category.

(2) Combinatorially (i.e. ignoring the ``labeling" of the vertices),
$Q_{\mathrm{tilt}}(A) = Q_{\mathrm{tilt}}(B)$ (respectively
$Q_{\mathrm{sms}}(A) = Q_{\mathrm{sms}}(B)$) if $A$ and $B$ are
derived (resp. stably) equivalent.

\end{Rem}

\begin{Prop}\label{mutation-transitive-sms}
Suppose $A$ is an RFS algebra.  Then there is a surjective map
$Q_{\mathrm{tilt}}(A)\to Q_{\mathrm{sms}}(A)$.  In particular, every
sms of $A$ can be obtained by iterative (left irreducible) mutation
starting from the simple $A$-modules.
\end{Prop}
\begin{Proof}
Using the correspondence, which respects mutation, between
(Nakayama-stable) silting object and (Nakayama-stable) smc in
\cite{KY}, the vertices of $Q_{\mathrm{tilt}}(A)$ can be identified
with Nakayama-stable smc. Every sms of $A$ is liftable simple-image
(see Proof of Theorem \ref{RFS-sms-smc-bijec}). This implies
surjectivity on the set of vertices. The surjectivity on the set of
arrows now follows from a result of Dugas \cite[Proposition 5.4]{Dugas2012}. For the last statement, let $\mathcal{S}$ be an sms
of $A$, then $\mathcal{S}$ is liftable to a Nakayama-stable smc
$\mathbf{S}$, which corresponds to a tilting object $T$. By Theorem
\ref{RFS-tilting-connected}, we can obtain $T$ by iterative tilting
mutations starting from $A$. The bijection in \cite{KY} then implies
that $\mathbf{S}$ can be obtained by iterative smc mutations starting
from simple $A$-modules.  Finally, Dugas' result is applied to
restrict smc mutations to sms mutations.
\end{Proof}

Since the sms's of an RFS algebra are in
general not acted upon transitively by the stable Picard group,
this result shows that a mutation of sms's usually cannot be realized
by a stable self-equivalence.

This result can also be compared with Theorem \ref{sms-smc}, where
we formed the quotient of the class of all smc's (respectively
sms's) by the derived (respectively stable) Picard group, obtaining
an injection regardless of representation-finiteness.  On the other
hand, these quivers visualise how we can ``track" simple-image sms's
of Morita type, and they contain more structure than the sets
considered in Theorem \ref{sms-smc}. Yet it is still unclear how
these links between smc's (hence tilting complexes) and sms's can be
used to extract information about derived and/or stable Picard
groups.

Another connection of this kind, with two-term tilting complexes,
will be discussed in \cite{ChanInPrep}.
\bigskip

\subsection{Proof of Theorem \ref{RFS-tilting-connected} \`a la Aihara}
We use the notation $\mathcal{T}=K^b(\mathrm{proj}A)$ with $A$ an
RFS algebra over a field. The term tilting object refers to objects
in $\mathcal{T}$, that is, to complexes. Recall the following
notation from \cite{AI} and \cite{Aihara2012}.
\begin{Def}
Let $T,U$ be tilting objects of $\mathcal{T}$, write $T\geq U$ if
$Hom_{\mathcal{T}}(T,U[i])=0$ for all $i>0$.
\end{Def}
Note this defines a partial order on the class of silting (and
hence, tilting) objects of $\mathcal{T}$.  Applying Lemma
\ref{nakayama-stable-mutation} to \cite[Prop 2.24]{AI} yields:
\begin{Prop}
{Let $T,U$ be tilting objects} of a self-injective algebra, and $U_0=U=\nu U$ such that $T\geq U$ , then there are triangles
\[\xymatrix@R=0.2cm{
U_1\ar[r]^{g_1}&T_0\ar[r]^{f_0}&U_0\ar[r]&U_1[1],\\
&\cdots,\\
U_\ell\ar[r]^{g_\ell}&T_{\ell-1}\ar[r]^{f_{\ell-1}}&U_{\ell-1}\ar[r]&U_\ell[1],\\
0\ar[r]^{g_{\ell+1}}&T_\ell\ar[r]^{f_\ell}&U_\ell\ar[r]&0, }\] for
some $\ell\ge 0$ such that $f_i$ is a minimal right
$\mathrm{add}T$-approximation, $g_{i+1}$ belongs to the Jacobson
radical $J_{\mathcal{T}}$,  $\nu U_i=U_i$ and $\nu T_i= T_i$, for
any $0\le i\le \ell$.
\end{Prop}
\begin{proof}
The only difference of the proof here and the one in \cite{AI} is to
use Lemma \ref{nakayama-stable-mutation} on the triangles in the
proof.  Apply Nakayama functor to the triangle in the proof:
\[
\nu U_1\xrightarrow{\nu g_1} \nu M_0 \xrightarrow{\nu f_0} \nu N_0
\to \nu N_1[1]
\]
and applying Lemma \ref{nakayama-stable-mutation} again, this
triangle becomes
\[
\nu U_1\xrightarrow{\nu g_1} M_0\xrightarrow{f_0} N_0 \to \nu N_1[1]
\]
and by the axioms of triangulated category, $\nu U_1\cong U_1$.  Now
the proof continues as in \cite{AI}.
\end{proof}

This can be used to deduce the Nakayama-stable analogue of
\cite[Theorem 2.35, Prop 2.36]{AI}:
\begin{Thm}\label{thm-tiltposet}
Let $T,U$ be tilting objects of a self-injective algebra.  Then
\begin{enumerate}
\item If $T>U$, then there exists an irreducible left tilting mutation $P$ of
$T$ such that $T>P\geq U$.
\item The following are equivalent:
\begin{enumerate}
\item $U$ is an irreducible left tilting mutation of $T$;
\item $T$ is an irreducible right tilting mutation of $U$;
\item $T>U$ and there is no $P$ tilting such that $T>P>U$.
\end{enumerate}
\end{enumerate}
\end{Thm}
\begin{proof}
Proof of (1) is the same as the proof of \cite[Prop 2.12]{Aihara2012}, except that now we take a $\nu$-stable summand of
$T_\ell$ instead of an indecomposable summand.  Proof of (2) is the
same as the proof of \cite[Theorem 2.35]{AI}, without any change.
\end{proof}

We modify the proof of Aihara in \cite{Aihara2012} to show that any
tilting object of an RFS algebra can be obtained through iterative
irreducible tilting mutation.

The proof of Theorem \ref{RFS-tilting-connected} is based on the
following key proposition:
\begin{Prop}{\cite[Prop 5.1]{Aihara2012}}
$\mathcal{T}$ is tilting-connected if, for any algebra $B$ derived
equivalent to $A$, the following conditions are satisfied:
\begin{description}
\item[(A1)]  Let $T$ be a basic tilting object in $K^b(\mathrm{proj}{B})$ with
$B[-1]\geq T\geq B$.  Then $T$ is connected to both $B[-1]$ and $B$.
\item[(A2)]  Let $P$ be a basic tilting object in $K^b(\mathrm{proj}{B})$ with
$B[-\ell]\geq P\geq B$ for a positive integer $\ell$.  Then there
exists a basic tilting object $T$ in $K^b(\mathrm{proj}{B})$
satisfying $B[-1]\geq T \geq B$ such that $T[-\ell +1]\geq P\geq T$.
\end{description}
\end{Prop}

Since we are only interested in tilting-connectedness rather than
silting-connectedness, the original condition (A3), which says that
any silting object is connected to a tilting object, is discarded.

(A2) is known to be true from \cite[Lemma 5.4]{Aihara2012}.
Therefore, what is left is to look carefully at the arguments and
results that are used by Aihara in the proof of (A1).

\begin{Lem}{\cite[Lemma 5.3]{Aihara2012}}
Condition (A1) holds for all RFS algebras $A$.
\end{Lem}
\begin{proof}
The original proof relies on \cite[Prop 2.9]{Aihara2012} and
\cite[Theorem 3.5]{Aihara2012}.  Proposition 2.9 is true regardless
of what kind of algebra $A$ is.  We are left to show the analogue of
\cite[Theorem 3.5]{Aihara2012} is true, i.e. the
following:
\end{proof}

\begin{Thm}{\cite[Theorem 3.5]{Aihara2012}}\label{thm-tiltconnected}
Let $T, U$ be basic tilting objects in $\mathcal{T}$ with $T\geq U$.
If there exist only finitely many tilting objects $P$ such that
$T\geq P\geq U$, then $U$ is left-connected to $T$.
\end{Thm}
\begin{proof}
If $U\in \mathrm{add} T$, then we have $U\cong T$.  So suppose
$U\notin\mathrm{add}T$. Theorem \ref{thm-tiltposet} provides a
sequence:
\[
T=T_0 > T_1 > T_2 >\cdots
\]
such that each $T_{i+1}$ is an irreducible left tilting mutation of
$T_i$, and $T_i\geq U$, for all $i\geq 0$.  If $U$ is not
left-connected to $T$, then this sequence is infinitely long,
contradicting the condition that there are only finitely many
tilting objects $P$ with $T\geq P\geq U$.  Therefore, $U$ is
isomorphic to $T_i$ for some $i\geq 0$.
\end{proof}


\bigskip

\begin{thebibliography}{88}


\bibitem{Aihara2012}{{\sc T.Aihara,} Tilting-connected symmetric algebras.
 Algebr. Represent. Theory, 2012, DOI: 10.1007/s10468-012-9337-3.}

\bibitem{AI}{{\sc T.Aihara and O.Iyama,} Silting mutation in triangulated
categories. arXiv:1009.3370, 2010.}

\bibitem{Al-Nofayee2007}{{\sc S.Al-Nofayee,} Equivalences of derived
categories for self-injective algebras. J. Algebra  \textbf{313}
(2007), 897--904.}



\bibitem{Asashiba1999}{{\sc H.Asashiba,} The derived equivalence
classification of representation-finite self-injective algebras. J.
Algebra  \textbf{214} (1999), 182--221.}

\bibitem{Asashiba2003}{{\sc H.Asashiba,} On a lift of an individual stable
equivalence to a standard derived equivalence for
representation-finite self-injective algebras. Algebras and
Representation Theory \textbf{6} (4) (2003), 427--447.}

\bibitem{ARS}{{\sc
M.Auslander, I.Reiten and S.O.Smal\o,} {\it Representation theory of
Artin algebras}. Cambridge University Press, 1995.}

\bibitem{ASS}{{\sc
I.Assem, D.Simson, and A.Skowro\'{n}ski,} {\it Elements of the
Representation Theory of Associative Algebras Volume 1 Techniques of
Representation Theory}. London Mathematical Society Student Texts
{\bf 65}. Cambridge University Press, 2006.}

\bibitem{BG}{{\sc K.Bongartz and P.Gabriel,}
Covering spaces in representation theory. Invent. Math. \textbf{65}
(1982), 331--378.}

\bibitem{BLR}{{\sc O.Bretscher, C.L\"{a}ser, and C.Riedtmann,}
Self-injective ans simply connected algebras. Manuscripta Math.
\textbf{36} (1981), 253--307.}

\bibitem{Broue1994}{{\sc M.Brou\'{e},} Equivalences of blocks of group
    algebras. In: {\it Finite
dimensional algebras and related topics}. V.Dlab and L.L.Scott
(eds.), Kluwer, 1994, 1--26.}

\bibitem{ChanInPrep}{{\sc A.Chan,} Two-term tilting complexes of Brauer star
algebra and simple-minded systems.  arXiv:1304.5223, 2013.}

\bibitem{Dugas2011}{{\sc A.Dugas,} Tilting mutation of weakly symmetric
algebras and stable equivalence. arXiv:1110.1679, 2011.}

\bibitem{Dugas2012}{{\sc A.Dugas,} Torsion pairs and simple-minded systems in
triangulated categories. arXiv:1207.7338v1, 2012.}

\bibitem{Dugas}{{\sc A.Dugas,} Resolutions of mesh algebras: periodicity and
Calabi-Yau dimensions. Math. Z. \textbf{271} (2012), 1151--1184}


\bibitem{KL}{{\sc S.Koenig and
Y.Liu,} Simple-minded systems in stable module categories. Quart. J.
Math. \textbf{63} (3) (2012), 653--674.}

\bibitem{KY}{{\sc S.Koenig and
D.Yang,} On tilting complexes providing derived equivalences that
send simple-minded objects to simple objects, arXiv:1011.3938. \\
Silting objects, simple-minded collections, $t$-structures and
co-$t$-structures for finite-dimensional algebras. arXiv:1203.5657.}

\bibitem{Linckelmann1996}{{\sc M.Linckelmann,}
Stable equivalences of Morita type for self-injective algebras and
p-groups. Math. Zeit. \textbf{223} (1996), 87--100.}

\bibitem{Linckelmann1998}{{\sc M.Linckelmann,}
On stable equivalences of Morita type. In: {\it Derived equivalences
for group rings.} S.Koenig and A.Zimmermann (Eds.), LNM 1685, 1998,
221--232.}

\bibitem{Liu2003}{{\sc Y.M.Liu,} On stable
equivalences of Morita type for finite dimensional algebras. Proc.
Amer. Math. Soc. \textbf{131} (2003), 2657--2662.}

\bibitem{LX}{{\sc
Y.M.Liu and C.C.Xi,} Constructions of stable equivalences of Morita
type for finite dimensional algebras III. J. London Math. Soc.
\textbf{76} (3) (2007), 567--585.}

\bibitem{Martinez-Villa}{{\sc R.Martinez-Villa,}
The stable equivalence for algebras of finite representation type. Comm. in Algebra
\textbf{13} (5) (1985), 991--1018.}

\bibitem{Mizuno}{{\sc Y.Mizuno,} $\nu$-stable $\tau$-tilting modules.
arXiv:1210.8322, 2012.}

\bibitem{Rickard1989}{{\sc
J.Rickard,} Derived categories and stable equivalence. J. Pure Appl.
Algebra \textbf{61}(3) (1989), 303--317.}

\bibitem{Rickard1991}{{\sc J.Rickard,}
Derived equivalences as derived functors. J. London Math. Soc
\textbf{43} (1991), 37--48.}

\bibitem{Rickard2002}{{\sc J.Rickard,} Equivalences of derived
categories for symmetric algebras. J. Algebra  \textbf{257} (2002),
460--481.}

\bibitem{Riedtmann1}{{\sc Ch.Riedtmann,} Algebren Darstellungsk\"{o}cher,
\"{U}berlagerungen and zur\"{u}ck.
 Comment. Math. Helv. \textbf{55} (2) (1980), 199--224.}

\bibitem{Riedtmann2}{{\sc Ch.Riedtmann,} Representation-finite self-injective
algebras of class $A_n$.
 LNM 832, 1980, 449--520.}

\bibitem{Riedtmann3}{{\sc Ch.Riedtmann,} Configurations of $\mathbb{Z}D_n$.
 J. Algebra \textbf{82} (1983), 309--327.}

\bibitem{Riedtmann4}{{\sc Ch.Riedtmann,} Representation-finite self-injective
algebras of class $D_n$.
 Compositio Math. \textbf{49} (2) (1983), 231--282.}

\bibitem{Waschbusch1981}{{\sc J.Waschb\"{u}sch,} Symmetrische Algebren vom
endlichen Modultyp.
 J. Reine Angew. Math. \textbf{321} (1981), 78--98.}
\end{thebibliography}
\end{document}